\documentclass[12pt]{amsart}
\usepackage[utf8]{inputenc}
\usepackage{amsmath}
\usepackage{amssymb}
\usepackage{tikz}
\usepackage{color}
\usepackage{xcolor}
\usepackage{float}
\usepackage{caption}
\usepackage{wasysym}

\usepackage{mathrsfs}
\usepackage{makecell}
\usepackage{graphicx, amsmath,amssymb,amsthm,amsfonts, mathtools}
\usepackage{marginnote}

\usepackage{geometry}
\usepackage{cite}

\makeatletter
\@namedef{subjclassname@2020}{%
  \textup{2020} Mathematics Subject Classification}
\makeatother
 
\theoremstyle{definition}

\newtheorem{theorem}{Theorem}
\newtheorem{corollary}[theorem]{Corollary}
\newtheorem{lemma}[theorem]{Lemma}

\newtheorem{proposition}[theorem]{Proposition}
\newtheorem{definition}[theorem]{Definition}

\newcommand{\R}{\mathbb{R}}
\newcommand{\dd}{~\mathrm{d}}

\newcommand{\supp}{\mathrm{supp}}
\newcommand{\dm}{d_{\mathrm{{m}}}}

\newcommand{\wpx}{\mathcal{W}_p(X,\dm)}

\newcommand{\wprt}{\mathcal{W}_p(\mathbb{R}^2,\dm)}
\newcommand{\wpq}{\mathcal{W}_p(Q,\dm)}
\newcommand{\wpl}{\mathcal{W}_p(L,\dm)}

\newcommand{\wort}{\mathcal{W}_1(\mathbb{R}^2,\dm)}
\newcommand{\woq}{\mathcal{W}_1(Q,\dm)}

\title[Isometric rigidity of $\wprt$ and $\wpq$ for $p\geq1$]{Isometric rigidity of the Wasserstein space over the plane with the maximum metric}
\author[Zolt\'an M. Balogh]{Zolt\'an M. Balogh}
\address{Zolt\'an M. Balogh, Universit\"at Bern\\ Mathematisches Institut (MAI)\\ Sidlerstrasse 12\\ 3012 Bern\\ Schweiz}
\email{zoltan.balogh@unibe.ch}

\author[Gergely Kiss]{Gergely Kiss}
\address{Gergely Kiss, Corvinus University of Budapest, Department of Mathematics \\
Fővám tér 13-15 \\ Budapest 1093 \\ Hungary\\ and HUN-REN Alfr\'ed R\'enyi Institute of Mathematics\\ Re\'altanoda u. 13-15.\\
Budapest 1053\\ Hungary}\email{kiss.gergely@renyi.hu}

\author[Tam\'as Titkos]{Tam\'as Titkos}
\address{Tam\'as Titkos, Corvinus University of Budapest, Department of Mathematics \\
Fővám tér 13-15 \\ Budapest 1093 \\ Hungary\\ and HUN-REN Alfr\'ed R\'enyi Institute of Mathematics\\ Re\'altanoda u. 13-15.\\
Budapest 1053\\ Hungary}
\email{titkos.tamas@renyi.hu}

\author[D\'aniel Virosztek]{D\'aniel Virosztek}
\address{D\'aniel Virosztek, HUN-REN Alfr\'ed R\'enyi Institute of Mathematics\\ Re\'altanoda u. 13-15.\\
Budapest H-1053\\ Hungary}
\email{virosztek.daniel@renyi.hu}

\subjclass[2020]{Primary: 54E40; 46E27. Secondary: 60B05.}

\keywords{optimal transport, isometry group, Wasserstein spaces, branching spaces, maximum norm}

\thanks{Z. M. Balogh is supported by the Swiss National Science Foundation, Grant Nr. {200021\_228012}. G. Kiss is supported by J\'anos Bolyai Research Fellowship of the Hungarian Academy of Sciences and by the Hungarian National Research, Development and Innovation Office - NKFIH grants no. K146922, FK142993, Starting 150576. T. Titkos is supported by the Hungarian National Research, Development and Innovation Office - NKFIH grant no. K134944), by the Momentum program of the Hungarian Academy of Sciences under grant agreement no. LP2021-15/2021., and by the C-PAP Scholarship of Corvinus University. D. Virosztek is supported by the Momentum program of the Hungarian Academy of Sciences under grant agreement nr. {LP2021-15/2021}, by the Hungarian National Research, Development and Innovation Office (NKFIH) under grant agreement nr. {Excellence\_151232}, and partially supported by the ERC Synergy Grant No. 810115.}

\begin{document}

% Add a list of keywords. Only capitalise a keyword if it starts with a proper name.
%------

\begin{abstract}
We study $p$-Wasserstein spaces over the branching spaces $\R^2$ and $[-1,1]^2$ equipped with the maximum norm metric. We show that these spaces are isometrically rigid for all $p\geq1,$ meaning that all isometries of these spaces are induced by isometries of the underlying space via the push-forward operation. This is in contrast to the case of the Euclidean metric since with that distance the $2$-Wasserstein space over $\R^2$ is not rigid. Also, we highlight that the $1$-Wasserstein space is not rigid over the closed interval $[-1,1]$, while according to our result, its two-dimensional analog, the closed unit ball $[-1,1]^2$ with the more complicated geodesic structure is rigid.
\end{abstract}
	\maketitle 
	\tableofcontents
\section{Introduction and the main result}
Recent developments of optimal mass transport theory \cite{AG,AGS,AP,Figalli-Glaudo,V,Villani} serve as main motivation for studying the Wasserstein space; that is, the space of probability measures endowed with a metric generated by optimal mass transport. The structure of the isometry group of Wasserstein spaces has been studied for the first time in a groundbreaking paper by Kloeckner 
\cite{K} in the case when the underlying space is the Euclidean space $\R^n$. This research has been followed up by Bertrand and Kloeckner \cite{BK,BK2}, Geh\'er, Titkos, Virosztek \cite{GTV1,GTV2}, Santos-Rodriguez \cite{S-R}. These authors considered various underlying metric spaces with different properties. The general feature of these spaces was, that they were non-branching geodesic metric spaces. This non-branching property of the underlying space was inherited by the Wasserstein space as well \cite{AG,S-R} and it was used in an essential way (e.g. in \cite{S-R}) to show that isometries of Wasserstein spaces preserve the class of Dirac masses.  

In this paper, we consider the situation of branching spaces, namely $\mathbb{R}^2$ and $Q=[-1,1]^2$ endowed with the maximum metric. Since the above-mentioned technique does not work in our case, we shall use a different method in order to determine the structure of the isometry group of the Wasserstein space over these spaces. 

Before defining the necessary notions and introducing the notation we will use throughout this paper, we highlight a very recent result of Che, Galaz-García, Kerin, and Santos-Rodríguez \cite{S-R2} which provided interesting examples of non-rigid Wasserstein spaces over certain classes of normed spaces. \\

To state our main result we start by introducing some notation.  Let  $X\subseteq\mathbb{R}^2$ be a closed subset equipped with the maximum metric $\dm:X\times X\to [0,\infty)$ $$\dm\big((x_1, x_2), (y_1,y_2)\big)=\max\big\{|x_1-y_1|,|x_2-y_2|\big\},$$ 
which is a complete and separable metric space. For $p\geq 1$ we consider the $p$-Wasserstein space $\big(\mathcal{P}_p(X,\dm),d_{W_p}\big)$, where  $X \subseteq \mathbb{R}^2$ is a closed subset and  $\mathcal{P}_p(X,\dm)$ is the space of Borel probability measures $\mu$ supported on  $X \subseteq \mathbb{R}^2$ with finite $p$-th moments:
  $$\int\limits_{X} \dm^p(x, x_0) ~\mathrm{d}\mu(x) < \infty $$
  for some (and thus for all) $x_0 \in \mathbb{R}^2$. This set is endowed with the Wasserstein metric coming from optimal mass transport, i.e. 
  $$d_{W_p}(\mu, \nu)=\min_{\pi\in C(\mu,\nu)}\left(\iint\limits_{X\times X} \dm^p(x,y) ~\mathrm{d}\pi(x,y)\right)^{\frac{1}{p}},$$
where $C(\mu, \nu)$ is in the set of couplings between $\mu$ and $\nu$. That is, $\pi\in \mathcal{P}(X\times X)$ and its marginals are equal to $\mu$ and $\nu$: $\pi(A\times X) = \mu(A)$, and $\pi(X\times A) = \nu(A)$ for any Borel set $A\subseteq X$. Recall that if $0<p<1$, then the definition of the $p$-Wasserstein distance is slightly different. In that case, $d_{W_p}(\mu, \nu)=\min_{\pi\in C(\mu,\nu)}\iint\limits_{X\times X} \dm^p(x,y) ~\mathrm{d}\pi(x,y)$.

For the sake of brevity, we will denote the Wasserstein space $\big(\mathcal{P}_p(X,\dm),d_{W_p}\big)$ by $\wpx$.

The support of a measure $\mu$ will be denoted by $\supp(\mu)$. For some distinguished collection of lines $L\subset\mathbb{R}^2$ the set $$\wpl=\{\mu\in \wprt\,|\, \supp(\mu)\subseteq L\}$$ will play an important role. Recall that a geodesic segment (or shortly: geodesic) is a curve $\gamma: [a,b] \to \wpx$ such that 
$$d_{W_p}(\gamma(t),\gamma(s)) = C|t-s|$$ for all $t,s \in \mathbb{R}$. Note, that by reparametrising the curve $\gamma$ we can always achieve that $C=1$. Geodesics with $C=1$ will be called unit-speed geodesics.

This paper aims to connect isometries of the underlying space $X$, and the Wasserstein space $\wpx$. Recall that given a metric space $(M,\varrho)$ a map $f:M\to M$ is called an isometry if it is bijective and distance preserving, i.e. $\varrho(f(m),f(m'))=\varrho(m,m')$ for all $m,m'\in M$. 

Recall that any isometry of $(X, \dm)$ induces an isometry of $\wpx$ by push-forward. Indeed, if $T: X \to X$ is an isometry, then the map $T_{\#}$ is an isometry of $\wpx$, where $T_{\#} \mu$ stands for the push-forward measure of $\mu$ by $T$ 
$$T_{\#} \mu(A) = \mu(T^{-1}(A)\quad\mbox{for all Borel sets}\quad A \subseteq X.$$

In what follows we shall call isometries of the type $T_{\#}$ {\it trivial isometries}. We call the Wasserstein space $\wpx$ {\it isometrically rigid} if all of its isometries are trivial.

Let us recall that by the results of Kloeckner \cite{K} the quadratic Wasserstein space $\mathcal{W}_2(\mathbb{R}^n,d_{\|\cdot\|_2})$ is not rigid as it has non-trivial shape-preserving isometries. Moreover, in the case $n=1$ there is a flow of exotic (non-shape-preserving)  isometries. Furthermore, the structure of the isometry group of Wasserstein spaces could depend both on the choice of $X$ and the value of $p$. Indeed, the results of \cite{GTV1} show in the one-dimensional case  $X=\mathbb{R}$ that the isometry group of $\mathcal{W}_2(\mathbb{R},d_{|\cdot|})$ is much larger than the isometry group of  $\mathcal{W}_p(\mathbb{R},d_{|\cdot|})$ for all $p\neq2$, while if $X=[0,1]$, then  the isometry group of $\mathcal{W}_1([0,1],d_{|\cdot|})$ is richer than the isometry group of $\mathcal{W}_p([0,1],d_{|\cdot|})$ for all $p>1$ (see \cite{GTV1}). As it was already pointed out in \cite{GTV1}, the same conclusion holds for every compact interval $[a,b]$. For our considerations, the relevant conclusion is that $\mathcal{W}_p([-1,1],|\cdot|)$ is rigid if and only if $p\neq1$.

In this paper, we distinguish the cases $p=1$ and $p>1$. We note that the case $p<1$ has already been covered by the general result \cite[Corollary 4.7.]{GTV2} which says that the Wasserstein space $\mathcal{W}_p(X,d)$ is isometrically rigid for every Polish underlying space $(X,d)$ and for every parameter $p<1.$ Furthermore, the underlying space $X$ will be either $\mathbb{R}^2$ or the closed unit ball $Q= [-1,1]^2$. Our main result shows, that in contrast to the above non-rigidity results in the one-dimensional case, (and also in the higher dimensional $\R^n$ with the Euclidean metric) in our situation the Wasserstein spaces are isometrically rigid when the underlying space $\R^2$ or $Q$ is considered with the maximum metric.  

\begin{theorem}\label{T: main} Let $X= \R^2$ or $X= Q= [-1,1]^2$ equipped with the maximum metric. Then for any $p\geq 1$ the Wasserstein space $\wpx$ is isometrically rigid. That is, for any isometry $\Phi: \wpx \to \wpx$ there exists a unique isometry $T:(X,\dm)\to (X,\dm)$ such that 
$$ \Phi (\mu) = T_{\#} \mu , \ \text{ for any } \ \mu \in \wpx.$$
\end{theorem}

The proof will be a combination of Proposition \ref{Diagonal-Rigidity} with Theorem \ref{T: main-1}, Theorem \ref{T: main-2}, Theorem \ref{T: main-3}, and Theorem \ref{T: main-4}.\\

Due to the difficulty caused by the branching nature of the underlying space, instead of Dirac masses, we shall consider measures supported on diagonal lines and prove that this class of measures is preserved by isometries. This seems to be a similar phenomenon to the one in the recent work of Balogh, Titkos, and Virosztek  \cite{BTV} about rigidity in the setting of the Heisenberg group. In that paper, the authors proved that measures supported on vertical lines in the Heisenberg group are preserved. In our setting vertical lines will be replaced by diagonals that are suitable to our geometry. In the sequel, we shall consider two special lines (briefly: diagonals)
$$L_{+}=\{(t,t)\,|\, t\in \mathbb{R}\} \quad \text{and} \quad  L_{-}=\{(t, -t)\,|\, t\in \mathbb{R}\},$$
and their translates:
\begin{equation} \label{eq:diag-line-def}
L= L_{\varepsilon, a} =\left\{ (x_1, x_2) \in \R^2 \, \middle| \, x_2= \varepsilon x_1 +a \right\} \text{ for some } \varepsilon \in \{-1,1\} \text{ and } a \in \R.
\end{equation}
When we are working in $Q=[-1,1]^2$, (with a slight abuse of notation) these symbols denote the line segment contained in $Q$.

The consideration of diagonal lines in our arguments is based on the observation that there is a unique geodesic (with respect to the maximum metric) connecting two points in the plane if and only if the two points are on the same diagonal $L_{\varepsilon, a}$.  We think that understanding rigidity in this special branching space will give us important clues to tackle the same question in general normed spaces.

\begin{definition} Let $X$ be either $\R^2$ or $Q$ equipped by the maximum metric. We call the Wasserstein space $\wpx$ {\it diagonally rigid}, if for every Wasserstein isometry $\Phi:\wpx\to\wpx$ there exists an isometry $T:(X,\dm)\to(X,\dm)$ such that $\Phi(\mu)=T_\#(\mu)$ whenever the support of $\mu$ is a subset of $L_+$ or $L_{-}$. 
\end{definition}

Our aim in Section \ref{s:diagonal-implies-full} is to prove that diagonal rigidity implies rigidity. This step is quite general in the sense that its proof works the same way for any $p\geq 1$ and $X\in \{\R^2,Q\}$.  To prove that $\wpx$ is indeed diagonally rigid is more tricky and uses very different arguments for different underlying spaces $X= \R^2$ and $X= [-1,1]^2$ and for different values of $p$. These results are proven in Section \ref{s: main}.  

\section{Diagonal rigidity implies rigidity} \label{s:diagonal-implies-full}

The main result of this section is the following:

\begin{proposition} \label{Diagonal-Rigidity}
Let $p\geq1$ and $X\in\{\R^2,Q\}$. Assume that the space $\wpx$ is diagonally rigid. Then the $\wpx$ is rigid.
\end{proposition}
For the sake of brevity, we only prove the case $X=\mathbb{R}^2$. The same argument works in the case $X=Q$, replacing lines by line segments contained in $Q$.

The proof of the statement will be a combination of lemmas. The first lemma is about the minimal distance projection onto lines. 
\begin{lemma}
Let $L\subset\R^2$ be a line that is not parallel to the $x$-axis and the $y$-axis, and let $x\in \R^2$. Then there exists unique $\hat{x}\in L$ such that $\dm(x,\hat{x})\le \dm(x,y)$ for all $y\in L$. 
\end{lemma}

\begin{proof}
If $x\in L$, then we take $\hat{x}:=x$ and the claim is obvious. If $x\not\in L$, then $\hat{x}$ is the first point of contact of metric balls centered at $x$ with $L$. 
Since metric balls are squares aligned with the $x$ and $y$-axes, and by assumption $L$ is not parallel to any of these axes, $\hat{x}$ is uniquely defined. 

\end{proof}

Denoting by $P_L(x)=\hat{x}$, we obtain a well-defined projection map $P_L:\R^2\to L$. 
 Our second statement is about the projection of measures defined by the push forward under this projection map. 
\begin{lemma} \label{lemma:meas-proj}
Let $\mu \in \wprt$. Then,  the measure $\hat{\mu}={P_L}_{\#}(\mu)$ is the metric projection $\mathcal{W}_p(\mathbb{R}^2,\dm) \to\mathcal{W}_p(L,\dm)$ i.e. the unique measure in $\wpl$ such that $$d_{W_p}(\mu, \hat{\mu})\le d_{W_p}(\mu, \nu)$$
for all $\nu\in \wpl$.
\end{lemma}

\begin{proof}
To prove the inequality in the statement, let $\nu\in \wpl$ be an arbitrary measure and $\pi$ be an optimal coupling between $\mu$ and $\nu$. Then $$d_{W_p}^p(\mu, \nu)=\int\limits_{\R^2\times L}\dm^p(x,y)~\mathrm{d}\pi(x,y).$$
Since $\dm^p(x,y)\ge \dm^p(x, P_L(x))$ for all $x\in \R^2, y\in L$ and $\supp(\pi)\subseteq\R^2\times L$,  we have that 
\begin{align*}
\begin{split}
d_{W_p}^p(\mu, \nu)&=\int\limits_{\R^2\times L}\dm^p(x,y)~\mathrm{d}\pi(x,y) \ge \int\limits_{\R^2\times \R^2}\dm^p(x, P_L(x))~\mathrm{d}\pi(x, y)\\&=\int\limits_{\R^2}\dm^p(x, P_L(x))~\mathrm{d}\mu(x)=d_{W_p}^p(\mu , \hat{\mu}).
\end{split}
\end{align*}
This shows that $\hat{\mu}$ is a minimizer for the problem $\inf\{d_{W_p}^p(\mu, \nu)\colon \nu\in \wpl\}$.

To show that $\hat{\mu}$ is the unique minimizer, note that in the case equality we have that 
$y= P_L(x)$ for $\pi$ almost every $(x,y)$ showing that $\pi = (Id \times P_L)_{\#} \mu$ and thus $\nu= P_{L\#} \mu$.
\end{proof}

The next lemma shows that the action of the isometry and the push-forward by projection commute. 
\begin{lemma} \label{commutation} 
If $\Phi: \wprt \to \wprt$ is an isometry such that $\Phi(\mu)=\mu$ for all $\mu\in \mathcal{W}_p(L_+,\dm)\cup\mathcal{W}_p(L_-,\dm)$ then we have the commutation relations
$$
\Phi({P_{L_+}}_{\#}(\mu))={P_{L_+}}_{\#}(\Phi(\mu)) \qquad\text{ and }\qquad \Phi({P_{L_-}}_{\#}(\mu))={P_{L_-}}_{\#}(\Phi(\mu))
$$
 for all $\mu \in \wprt.$
\end{lemma}

\begin{proof}
The proof is based on the previous lemma, and we prove only the first commutation relation regarding $L_+$ as the case of $L_-$ is very similar. 
\par 
Let $\mu\in \wprt$ and $\hat{\mu}={P_{L_+}}_{\#}(\mu)$. We have to show that $\Phi(\hat{\mu})={P_{L_+}}_{\#}(\Phi(\mu))$.
Since $\hat\mu\in \mathcal{W}_p(L_+,\dm)$, we note that $\Phi(\hat\mu) = \hat\mu$ by assumption. 
As $\Phi$ is an isometry, $$D:=d_{W_p}(\mu, \hat{\mu})=d_{W_p}(\Phi(\mu), \Phi(\hat{\mu}))=d_{W_p}(\Phi(\mu), \hat{\mu}) .$$
Let $\nu\in \mathcal{W}_p(L_+,\dm)$. Thus $ \nu =\Phi^{-1}(\nu)\in \mathcal{W}_p(L_+,\dm),$ and therefore 
$$
d_{W_p}(\Phi(\mu), \nu)=d_{W_p}(\mu, \Phi^{-1}(\nu)) =d_{W_p}(\mu, \nu) \ge D \ \ \text{ for all}\ \nu\in \mathcal{W}_p(L_+,\dm).
$$
Since $d_{W_p}(\Phi(\mu), \Phi(\hat{\mu}))=D$ and $\hat{\mu}=\Phi(\hat{\mu})\in \mathcal{W}_p(L_+,\dm)$ is the minimizer of the distance, from the uniqueness part of Lemma \ref{lemma:meas-proj}, we have $\hat{\mu}= \Phi(\hat{\mu})={P_{L_+}}_{\#}(\Phi(\mu))$ as required. 
\end{proof}

After this preparation we can turn to the proof of Proposition \ref{Diagonal-Rigidity}. The proof is inspired by Bertrand and Kloeckner \cite{BK} and it is based on the method of Radon transform. In our case the Radon transform will be a mapping $\mathcal{R}: \wprt\to\mathcal{W}_p(L_+,\dm)\times\mathcal{W}_p(L_-,\dm)$ defined by
\begin{equation}
     \mathcal{R}(\mu) :=\Big( {P_{L_{+}}}_{\#}(\mu), {P_{L_{-}}}_{\#}(\mu)\Big).
\end{equation}

\begin{proof}[Proof of Proposition \ref{Diagonal-Rigidity}]
Without loss of generality we can assume that $\Phi(\mu)=\mu$ for all $\mu\in\mathcal{W}_p(L_{+},\dm)\cup \mathcal{W}_p(L_{-},\dm)$.
Now  we want to extend this property to the whole $\wprt$. The main idea is to consider a subset $\mathcal{F}\subset \wprt$ such that
\begin{enumerate}
    \item[$\bullet$] $\mathcal{F}$ is dense in  $\wprt$, 
    \item[$\bullet$] for any $\mu_1, \mu_2\in \mathcal{F}$ the following holds:
    $${P_{L_{+}}}_{\#}(\mu_1)={P_{L_{+}}}_{\#}(\mu_2)\quad\mbox{and}\quad {P_{L_{-}}}_{\#}(\mu_1)={P_{L_{-}}}_{\#}(\mu_2) \qquad\Longrightarrow\qquad \mu_1=\mu_2,$$
    \item[$\bullet$] $\Phi(\mathcal{F}) \subseteq \mathcal{F}$.
    \end{enumerate}
    The second condition is the injectivity of the Radon transform on the set $\mathcal{F}$.   
    Suppose that we have such an $\mathcal{F}$. Then, applying Lemma \ref{commutation} we get that for any $\mu\in \mathcal{F}$  $${P_{L_{+}}}_{\#}(\Phi(\mu))=\Phi({P_{L_{+}}}_{\#}(\mu))={P_{L_{+}}}_{\#}(\mu)$$
    and
    $${P_{L_{-}}}_{\#}(\Phi(\mu))=\Phi({P_{L_{-}}}_{\#}(\mu))={P_{L_{-}}}_{\#}(\mu).$$
  By the third condition we have $\Phi(\mu) \in \mathcal{F}$ and so we can apply the second condition for the two measures $\mu_1 = \mu$ and $\mu_2 = \Phi(\mu)$. This implies that $\Phi(\mu)=\mu$ for all $\mu\in \mathcal{F}$. Using the density of $\mathcal{F}$ in  $\wprt$ (the first condition) we get that $\Phi(\mu)=\mu$ for all $\mu\in\mathcal{W}_p(\R^2,\dm)$. 

    Therefore it is enough to find a set  $\mathcal{F}$ that satisfies the conditions above.
    We define $\mathcal{F}$ by the following: 
\begin{align*}
\mathcal{F}:=&\Bigg\{\sum_{i=1}^N a_i\delta_{x_i}\,\Bigg| N\ge 1, \sum_{i=1}^Na_i = 1, \\ & \textrm{for} \ i\ne j \ \textrm{we require} \ a_i\ne a_j, \ \textrm{and} \ P_{L_+} x_i\ne P_{L_+} x_j \textrm{ and } P_{L_-} x_i\ne P_{L_-} x_j  \Bigg\}.
\end{align*}

Let us check the required conditions for this choice of $\mathcal{F}$. For the first condition we use the fact that the set of finitely supported measures is dense in $\wprt$. Since a finitely supported measure can be clearly approximated in $\wprt$ by elements of $\mathcal{F}$, the first property follows. 

In order to show the second property let $\mu_1, \mu_2 \in \mathcal{F}$ such that $P_{L_{+}}{}_{\#}(\mu_1) = P_{L_{+}}{}_{\#}(\mu_2)$ and $P_{L_{-}}{}_{\#}(\mu_1) = P_{L_{-}}{}_{\#}(\mu_2)$. We have to conclude that $\mu_1= \mu_2$. 

To check this, let us assume that 
$$ \mu_1 = \sum_{i=1}^{N_1} a_i^{(1)} \delta_{x_i^{(1)}}, \ \textrm{and}  \ \mu_2 = \sum_{i=1}^{N_2} a_i^{(2)} \delta_{x_i^{(2)}}.$$ 
By the condition that $P_{L_{+}}{}_{\#}(\mu_1) = P_{L_{+}}{}_{\#}(\mu_2)$ and $P_{L_{-}}{}_{\#}(\mu_1) = P_{L_{-}}{}_{\#}(\mu_2)$ we obtain the equations 
$$\sum_{i=1}^{N_1} a_i^{(1)} \delta_{P_{L_{+}}(x_i^{(1)})} = \sum_{i=1}^{N_2} a_i^{(2)} \delta_{P_{L_{+}}(x_i^{(2)})} $$
and
$$\sum_{i=1}^{N_1} a_i^{(1)} \delta_{P_{L_{-}}(x_i^{(1)})} = \sum_{i=1}^{N_2} a_i^{(2)} \delta_{P_{L_{-}}(x_i^{(2)})}. $$
From here we conclude, that  $N_1=N_2= N$, $a_i^{(1)}=a_i^{(2)}$ and $x_i^{(1)}= x_i^{(2)}$ for $i= 1,\ldots, N$  which gives that $\mu_1=\mu_2$.

To verify the third property, i.e. that $\Phi(\mathcal{F})\subseteq \mathcal{F}$, let us take an element $\mu =\sum_{i=1}^N a_i \delta_{x_i} \in \mathcal{F}$. Recalling that $\Phi$ fixes all measures supported on the diagonals $L_+$ and $L_-$, we get by Lemma \ref{commutation} 
 that 
$$ {P_L}_{\#}(\mu)=\Phi({P_L}_{\#}(\mu))={P_L}_{\#}(\Phi(\mu))$$
for $L\in\{L_+,L_{-}\}$, and therefore we have 
$${P_{L_+}}_{\#}(\Phi(\mu))= {P_{L_+}}_{\#}(\mu) = \sum_{i=1}^N a_i \delta_{P_{L_+}(x_i)} , $$
where $P_{L_+}(x_i) \in L_+$ and 
$$ {P_{L_-}}_{\#}(\Phi(\mu))= {P_{L_-}}_{\#}(\mu) = \sum_{i=1}^N a_i \delta_{P_{L_-}(x_i)} , $$
where $P_{L_-}(x_i)\in L_{-}$. In conclusion we obtain that the Radon transform of $\mu$ and $\Phi(\mu)$ are equal 
\begin{equation*}
\mathcal{R} :=\mathcal{R}(\mu)=\mathcal{R}(\Phi(\mu))= \left( \sum_{i=1}^N a_i \delta_{P_{L_+}(x_i)}, \sum_{i=1}^N a_i \delta_{P_{L_-}(x_i)} \right).
\end{equation*}

Now observe that $\Phi(\mu)$ is a finitely supported measure with support contained in the intersection of the two pre-images:
$$ (P_{L_+})^{-1}( \{ P_{L_+}(x_1), \ldots, P_{L_+}(x_N) \} ) \cap (P_{L_-})^{-1}( \{ P_{L_-}(x_1), \ldots, P_{L_-}(x_N) \} ).$$
This intersection is an $N$-by-$N$ grid, and we refer to its points by $z_{i,j}$ $(1\le i,j\le N)$,  
where $z_{i,j}$ satisfies that $P_{L_+}(z_{i,j})=P_{L_+}(x_i)$ and $P_{L_-}(z_{i,j})=P_{L_-}(x_j)$. Hence $\Phi(\mu)$ can be written as $$\Phi(\mu)=\sum_{i=1}^N\sum_{j=1}^N a_{i,j} \delta_{z_{i,j}},$$ where $a_{i,j}\ge 0$, $\sum_{i=1}^N\sum_{j=1}^N a_{i,j}=1$ and $\sum_{i=1}^N a_{i,j}=a_j$ and $\sum_{j=1}^N a_{i,j}=a_i$. Since the grid is finite, there is a positive minimal distance between its points
\begin{equation*}
    c := \min_{(i,j)\ne(i',j')}d_m(z_{i,j},z_{i',j'})>0.
\end{equation*}
From here, we assume by contradiction that $\Phi(\mu)\notin\mathcal{F}$.
Since $\Phi(\mu)\notin\mathcal{F}$, then there exist at least two points $z,z'\in\supp(\Phi(\mu))$ such that their projection onto either $L_+$ or $L_-$ coincide. 
\begin{figure}[H]
\centering
\includegraphics[width=0.7\textwidth]{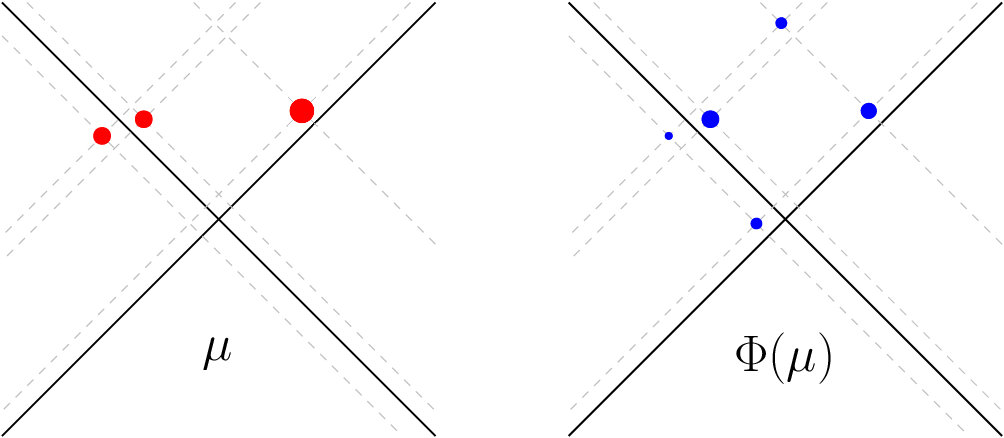}
\caption{Illustration of a finitely supported measure $\mu$ with a possible image $\Phi(\mu)$ and the grid determined by the pre-images of $P_{L_+}$ and $P_{L_-}$.}
\label{fig:mu-and-Phi(mu)}
\end{figure}
 We briefly sketch how to obtain the desired contradiction and we give the details later:

First, by slightly perturbing the measure $\mu$ we will construct a measure $\mu'$ such that
\begin{equation} \label{eq:argmin} \arg\min\big\{d_{W_p}(\mu,\xi)\,\big|\,\xi\in\mathcal{W}_p(\mathbb{R}^2,\dm), \mathcal{R}(\xi)=\mathcal{R}(\mu')\big\}=\{\mu'\}.
\end{equation} 

\begin{figure}[H]
\centering
\includegraphics[width=0.7\textwidth]{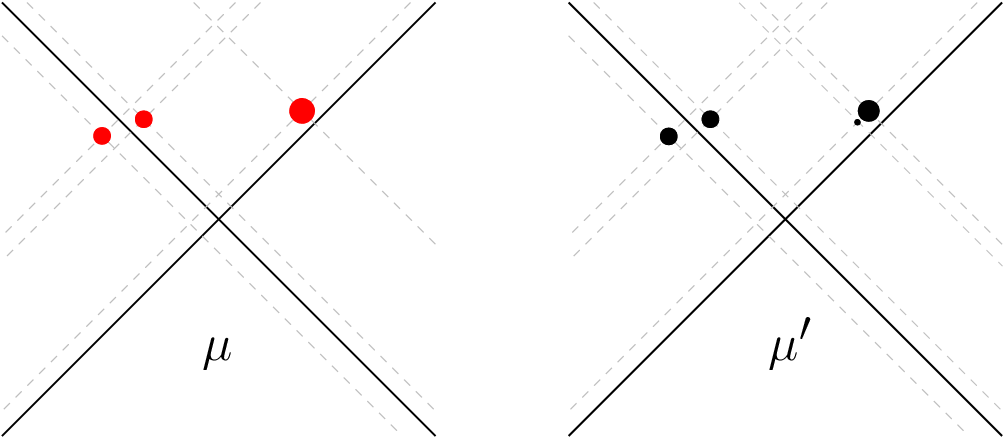}
\caption{Illustration of $\mu'$, the measure that we obtain by a sufficiently small perturbation of $\mu$.}
\label{fig:mu-and-muprime}
\end{figure}
Next,  using the existence of $z$ and $z'$ as above; by small perturbations of $\Phi(\mu)$ we will construct two measures $\nu_1'$ and $\nu_2'$ such that $\mathcal{R}(\nu_1')=\mathcal{R}(\nu_2')= \mathcal{R}(\mu')$ and  
\begin{equation} \label{eq:argmin2} \arg\min\big\{d_{W_p}(\Phi(\mu),\xi)\,\big|\,\xi\in\mathcal{W}_p(\mathbb{R}^2,\dm), \mathcal{R}(\xi)=\mathcal{R}(\mu')\big\} \supseteq \{\nu_1', \nu_2'\}.
\end{equation} 
    \begin{figure}[H]
\centering

\includegraphics[width=0.7\textwidth]{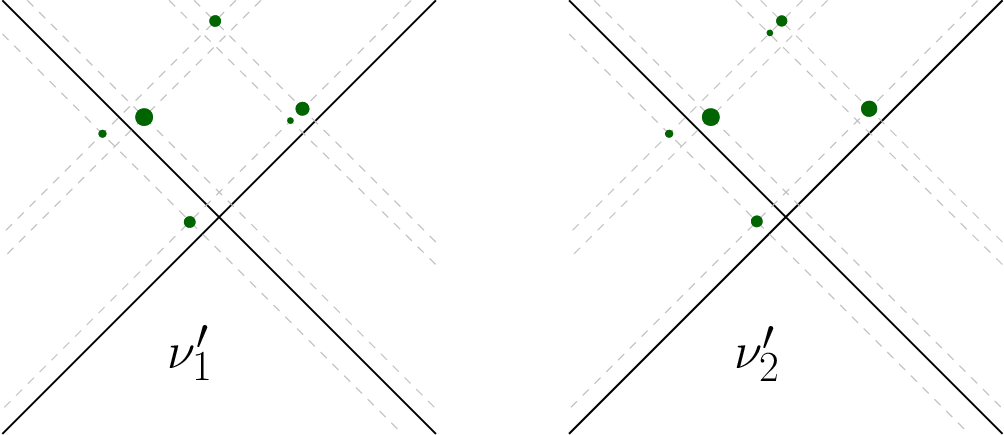}
\caption{Illustration for $\nu_1'$ and $\nu_2'$ - the two measures that we obtained by sufficiently small perturbations of $\Phi(\mu)$.}
\label{fig:nu1prime-nu2prime}
\end{figure}
Finally, $d_{W_p}(\mu,\Phi^{-1}(\nu_1'))=d_{W_p}(\mu,\mu')=d_{W_p}(\mu,\Phi^{-1}(\nu_2'))$ contradicts the fact that $\mu'$ is the unique minimizer. This contradiction guarantees that $\Phi(\mu)\in\mathcal{F}$.
   \begin{figure}[H]
   \centering

\includegraphics[width=0.7\textwidth]{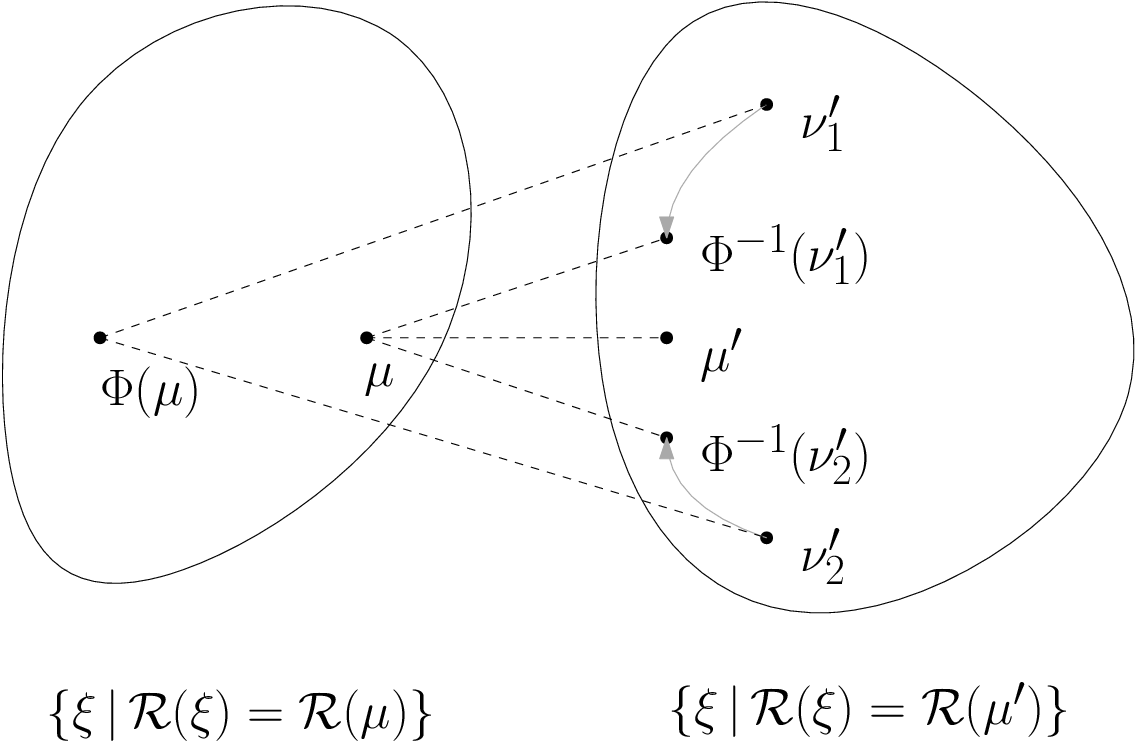}
\caption{Illustration of the final step leading to a contradiction. Dashed lines represent equal distances.}
\label{fig:krumpli}
\end{figure}

After this brief sketch of the proof, we turn to the details. 

If $\Phi(\mu) \not\in \mathcal{F}$, then there exist two points $z, z' \in \supp (\Phi(\mu))$ such that their projections onto either $L_+$ or $L_-$ coincide. Indeed, if there are no such points, then all points of the support of $\Phi(\mu)$ project to different points of $L_+$ and $L_-$. Since $${P_{L_-}}_{\#}(\Phi(\mu)) =  \sum_{i=1}^N a_i \delta_{P_{L_-}(x_i)}, \ \ {P_{L_+}}_{\#}(\Phi(\mu)) =  \sum_{i=1}^N a_i \delta_{P_{L_+}(x_i)}$$ and $a_i\ne a_j$ ($1\le i\ne j\le N$), this implies that $\Phi(\mu)=\sum_{i=1}^N a_i \delta_{x_i}=\mu$, which leads to a contradiction. Without loss of generality, we can assume that this common projection is $P_{L_+}(x_1)\in L_+$, i.e., $P_{L_+}(z)=P_{L_+}(z')=P_{L_+}(x_1)$, and for some $1\le j_1\ne j_2\le N$ we have $z=z_{1,j_1}, z'=z_{1,j_2}$ and $a_{1,j_1}>0, a_{1,j_2}>0$. Using this observation, we construct the measures $\mu',\nu_1',\nu_2'$ as follows. We take a point $x'\in L_+$ such that $c_0:=d_m(P_{L_+}(x_1),x')<c/2$.  Let us denote the elements of $$\Big((P_{L_+})^{-1}( \{ P_{L_+}(x') \} )\Big) \cap \Big((P_{L_-})^{-1}( \{ P_{L_-}(x_1), \ldots, P_{L_-}(x_N) \} )\Big)$$ by 
$z_{0,j}$ ($1\le j\le N$) so that $P_{L_+}(z_{0,j})=x'$ and $P_{L_-}(z_{0,j})=P_{L_-}(x_j)$. We will also use the notation $x_0=z_{0,1}$.
For every $0\le i'\le N,1\le j,j'\le N$  $$d_m(z_{0,j},z_{1,j})< d_m(z_{i',j'},z_{1,j})$$
if  $(i',j')\notin\{ (0,j),(1,j)\}$. To see this, observe that by construction, $d_m(z_{0,j},z_{1,j})= c_0$. If $i'\ne 0$ and $(i',j')\neq(1,j)$, then we have $d_m(z_{i',j'},z_{1,j})>c$ by definition. If $i'=0$ and $j'\ne j$, then using the reverse triangle inequality, we have
    $d_m(z_{0,j'}, z_{1,j})\ge d_m(z_{1,j'},z_{1,j})-d_m(z_{0,j'}, z_{1,j'})\ge c- c_0> c_0.$
Let us fix a weight $a$ satisfying $0<a< \min\{a_{1,j_1}, a_{1,j_2}\} <a_1$. 
Now, we consider the following measures
$$\mu'= a\delta_{x_0}+(a_1-a)\delta_{x_1}+\sum_{i=2}^N a_i \delta_{x_i},$$
$$\nu_1'= a\delta_{z_{0,j_1}}+(a_{1,j_1}-a)\delta_{z_{1,j_1}}+a_{1,j_2}\delta_{z_{1,j_2}}+\sum_{j=1, \\ j\ne j_1,j_2}^N a_{1,j}\delta_{z_{1,j}}+\sum_{i=2}^{N}\sum_{j=1}^N a_{i,j} \delta_{z_{i,j}},$$ 
$$\nu_2'= a\delta_{z_{0,j_2}}+a_{1,j_1}\delta_{z_{1,j_1}}+(a_{1,j_2}-a)\delta_{z_{1,j_2}}+\sum_{j=1, \\ j\ne j_1,j_2}^N a_{1,j}\delta_{z_{1,j}}+\sum_{i=2}^{N}\sum_{j=1}^N a_{i,j} \delta_{z_{i,j}}.$$
Obviously, $\mu', \nu_1'$, and $\nu_2'$ are probability measures satisfying $\mathcal{R}(\mu')=\mathcal{R}(\nu_1')=\mathcal{R}(\nu_2')=:\mathcal{R'}$, namely 
$$
\mathcal{R'}=\left( a\delta_{x'}+(a_1-a)\delta_{P_{L_+}(x_1)} + \sum_{i=2}^N a_i \delta_{P_{L_+}(x_i)}, \sum_{i=1}^N a_i \delta_{P_{L_-}(x_i)}\right).
$$
Our next step is to prove that: 
\begin{equation} \label{eq:meas-proj}
 d_{W_p}(\mu, \mu')=d_{W_p}(\Phi(\mu), \nu'_1)=d_{W_p}(\Phi(\mu),\nu'_2)=a^{\frac{1}{p}} c_0,  
\end{equation}
moreover, $\mu'$ satisfies the following uniqueness property: 
\begin{equation}\label{eq: uniq-prop} 
 \text{if} \  d_{W_p}(\mu, \xi)=a^{\frac{1}{p}} c_0, \ \text{and} \ \mathcal{R}(\xi)=\mathcal{R}' \ \text{then} \  \xi=\mu'.  
\end{equation}
Equations \eqref{eq:meas-proj} and \eqref{eq: uniq-prop} together will justify relations \eqref{eq:argmin} and \eqref{eq:argmin2}.

In order to show \eqref{eq:meas-proj} note that if $\xi_1,\xi_2$ are finitely supported probability measures with supports in a discrete set $P$, then $$d_{W_p}(\xi_1,\xi_2)=\min_{\Pi\in C(\xi_1, \xi_2)}\big(\sum_{(u,v)\in P \times P}d_m^p(u,v)\cdot \Pi(u,v)\big)^{\frac{1}{p}}.$$ 
The proof of each of the equations in \eqref{eq:meas-proj} is similar, therefore we will only prove the equality $d_{W_p}(\mu, \mu')=a^{\frac{1}{p}} c_0$. Notice first, that since every transport plan  must move a total weight of at least $a$ to $x_1$ from the support points of $\mu'$, we have that 
$$d_{W_p}(\mu,\mu')\ge \big( \sum_{i=0}^N d_m^p(x_i,x_1)\Pi(x_i, x_1)\big)^{\frac{1}{p}}\ge (d_m^p(x_0,x_1)a)^{\frac{1}{p}}=a^{\frac{1}{p}} c_0,$$ since $d_m^p(x_i,x_1)\ge d_m^p(x_0,x_1)$ and $\sum_{i=0}^N \Pi(x_i,x_1)\ge a$. On the other hand, if we move weight $a$ directly from $x_0$ to $x_1$, we exactly get that the cost of this transport plan is $a^{\frac{1}{p}} c_0$. 
Now we turn to the proof of \eqref{eq: uniq-prop}. Let us suppose that we have a probability measure $\xi$ with $\mathcal{R}(\xi)=\mathcal{R}'$. Then $\xi$ can be written in the form 
$$\xi=\sum_{i=0}^N\sum_{j=1}^N b_{i,j}\delta_{z_{i,j}},$$
such that $$\sum_{j=1}^N b_{0,j}=a,\ \ \sum_{j=1}^N b_{1,j}=a_1-a, \ \ \sum_{j=1}^N b_{i,j}=a_i \ (2\le i\le N),\ \ \sum_{i=0}^N b_{i,j}=a_j.$$
Again, every transport plan must move a total weight of at least $a$ to $x_1$ from the support points of $\xi$. Hence, we get that again
\begin{equation}\label{eqaeta}
d_{W_p}(\xi, \mu)\ge  \big( \sum_{i=0}^N\sum_{j=1}^N d_m^p(z_{i,j},x_1)\Pi(z_{i,j}, x_1)\big)^{\frac{1}{p}}\ge (d_m^p(x_0,x_1)a)^{\frac{1}{p}}=a^{\frac{1}{p}} c_0.
\end{equation}
Let us recall that $d_m(x_0,x_1)<d_m(z_{i,j},x_1)$, if $(i,j)\notin\{(0,1),(1,1)\}$. (Note that $z_{0,1}=x_0$, $z_{1,1}=x_1$.) Therefore, equality holds in \eqref{eqaeta} if and only if all transport occurs between $x_0$ and $x_1$ with weight $a$. This implies that $\xi-a\delta_{x_0}=\mu-a\delta_{x_1}$ and hence $\xi=\mu'$.

In the last step, we show that the existence of $\nu_1'\neq\nu_2'$ implies that $\mu'$ is not a unique minimizer in relation \eqref{eq:argmin}.  Indeed, since $\Phi^{-1}$ is an isometry preserving measures supported on $L_+$ and $L_-$ we have by Lemma \ref{commutation} that $\mathcal{R}(\Phi^{-1}(\nu_1'))= \mathcal{R}(\Phi^{-1}(\nu_2'))= \mathcal{R}'$. Furthermore, according to \eqref{eq:meas-proj}, we have
$$d_{W_p} (\mu, \Phi^{-1}(\nu'_1))=d_{W_p} (\Phi(\mu), \nu'_1)=a^{\frac{1}{p}} c_0=d_{W_p}(\mu,\mu'),$$
and similarly,
$$d_{W_p} (\mu, \Phi^{-1}(\nu'_2))=d_{W_p} (\Phi(\mu), \nu'_2)=a^{\frac{1}{p}} c_0=d_{W_p}(\mu,\mu').$$
Since $\Phi^{-1}(\nu_1')\neq\Phi^{-1}(\nu_2')$, this is a contradiction.
\end{proof}

\section{Proof of the main result}\label{s: main}

According to Proposition \ref{Diagonal-Rigidity}, it is enough to show that the Wasserstein space $\wpx$ is diagonally rigid. The proof of this fact is divided into four parts according to the choice of $X= \R^2$ or $X= Q$ and $p=1$ or $p>1$.

\subsection{Diagonal rigidity of $\wort$}

In this subsection, we deal with the case $p=1$ and show that $\wort$ is diagonally rigid. That is, we show that if $\Phi: \wort \to \wort$ is an isometry, then $\Phi(\mu)=\mu$ for all $\mu \in \mathcal{W}_1(L_{+}) \cup \mathcal{W}_1(L_{-})$ --- up to a trivial isometry induced by an isometry of the underlying space $\R^2.$

We recall the  slightly more general notion than $L_+$ and $L_-$ of diagonal lines  by calling  $L \subset \R^2$ a \emph{diagonal line} if
\begin{equation}
L= L_{\varepsilon, a} =\left\{ (x_1, x_2) \in \R^2 \, \middle| \, x_2= \varepsilon x_1 +a \right\} \text{ for some } \varepsilon \in \{-1,1\} \text{ and } a \in \R.
\end{equation}
Observe that these lines coincide with the set of images of $L_+$ by the isometry group of $(\R^2,d_m).$ 
The following proposition is a metric characterization of those elements of $\wort$ that are supported on a diagonal line. Let us note that this statement plays the same role as Lemma 3.5 in \cite{GTV2}, where Dirac masses were characterized in a similar way in Wasserstein spaces over a Hilbert space. In this sense, diagonally supported measures in our space have the same metric property as Dirac masses in the case of Hilbert spaces.

\begin{proposition} \label{prop:diag-line-supp-char}
Let $\mu \in \wort.$ The following statements are equivalent.
\begin{enumerate}
    \item[(i)] \label{crit:A} $\mu$ is supported on a diagonal line $L_{\varepsilon, a} \subset \R^2.$
    \item[(ii)] \label{crit:B} For every $\nu \in \wort$ there exists an $\eta \in \wort$ such that
    \begin{equation} \label{eq:char-eq}
    d_{W_1}(\mu, \nu)=d_{W_1}(\nu, \eta)=\frac{1}{2}d_{W_1}(\mu, \eta).
    \end{equation}
    In words, this item means that $\mu$ admits a
symmetrical with respect to every other measure.
\end{enumerate}
\end{proposition}

\begin{proof}
We prove the direction (i) $\Longrightarrow$ (ii) first. Let $\varepsilon \in \{-1,1\}$ and $a \in \R$ be fixed, let $\mu \in \wort$ such that 
$$
\mathrm{supp}(\mu) \subset L=L_{\varepsilon, a}=\left\{ (x_1, x_2) \in \R^2 \, \middle| \, x_2= \varepsilon x_1 +a \right\}.
$$ 
Let us construct the following map, which we will call the allocation of directions in the sequel:
\begin{equation} \label{eq:direction-alloc}
e: \R^2 \rightarrow \R^2; \, (y_1,y_2) \mapsto e((y_1,y_2)):=
\begin{cases}
(-\varepsilon, 1) & \text{if } y_2 \geq \varepsilon y_1+a \\
(\varepsilon, -1) & \text{if } y_2 < \varepsilon y_1+a    
\end{cases}
\end{equation}
See Figure \ref{fig:direction-alloc} for an illustration the map given above.
\begin{figure}[H]
\centering

\includegraphics[width=0.7\textwidth]{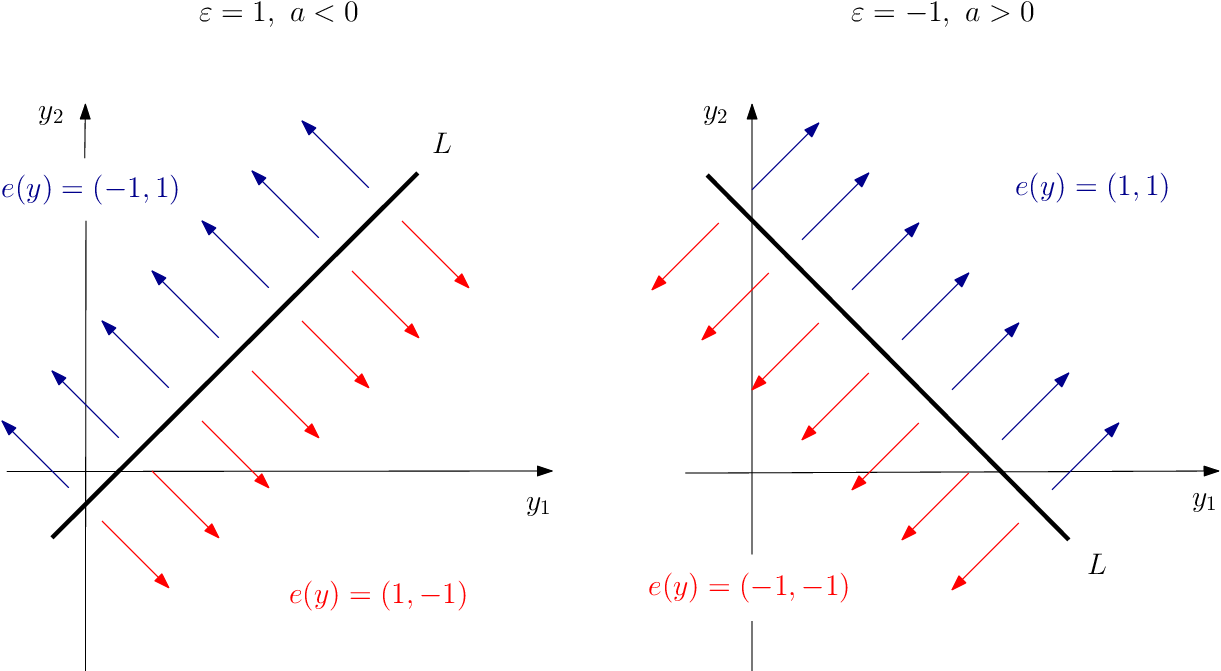}
\caption{The allocation of directions in $\R^2$ according to \eqref{eq:direction-alloc}}
\label{fig:direction-alloc}
\end{figure}
The above allocation of directions has the crucial property that for all $x=(x_1,x_2) \in L,$ for all $y=(y_1,y_2) \in \R^2,$ and for all $t \geq 0$ we have
\begin{equation} \label{eq:crucial-0}
\dm(x,y+t e(y))=\dm(x,y)+\dm(y, y+t e(y))=\dm(x,y)+t.
\end{equation}
Let us justify \eqref{eq:crucial-0} only in the sub-case $\varepsilon=1$ and $y_2 \geq \varepsilon y_1 +a$ as the other three sub-cases are very similar. We know that $x_2=x_1+a$ and $y_2 \geq y_1+a$ which implies that $y_2-x_2 \geq y_1 - x_1,$ or equivalently, $x_1-y_1 \geq x_2 - y_2.$ Therefore, 
\begin{equation} \label{eq:dm-exp-1}
\begin{split}
\dm(x,y)&=\dm((x_1,x_2),(y_1,y_2))\\
&=\max\{x_1-y_1, y_1-x_1,x_2-y_2,y_2-x_2\}\\
&=\max\{x_1-y_1,y_2-x_2\}.
\end{split}
\end{equation}
Moreover,
\begin{equation}\label{eq:dm-exp-2}
\begin{split}
\dm(x, y+t e(y))
&=\dm((x_1,x_2), (y_1,y_2)+t(-1,1))\\
&=\dm((x_1,x_2), (y_1-t,y_2+t))\\
&=\max\{x_1-y_1+t, y_1-x_1-t, y_2-x_2+t, x_2-y_2-t\}\\
&=\max\{x_1-y_1+t, y_2-x_2+t\}\\
&=\max\{x_1-y_1, y_2-x_2\}+t.
\end{split}
\end{equation}
That is, \eqref{eq:dm-exp-1} and \eqref{eq:dm-exp-2} shows that $\dm(x, y+t e(y))=\dm(x,y)+t$ indeed, and it is clear by the definition \eqref{eq:direction-alloc} that $\dm(y, y+ t e(y))=t$ for every non-negative $t.$
\par
It is a straightforward consequence of the definition of $e(y)$ --- see eq. \eqref{eq:direction-alloc} --- that for every $t \geq 0$ the map $y \mapsto y+t e(y)$ is an injection of $\R^2$ and hence invertible on its range.
\par
Let $\nu \in \wort$ and let $t_0:=d_{W_1}(\mu, \nu).$ Let us define
\begin{equation} \label{eq:eta-t-def}
\eta_t:=\left(y \mapsto y + t e(y) \right)_{\#} \nu 
\end{equation}
for $t \geq 0.$ As the map $y \mapsto y+t e(y)$ is invertible, the couplings of $\mu$ and $\nu$ are in a one-by-one correspondence with the couplings of $\mu$ and $\eta_t$ (for every $t\geq 0$), and this correspondence is given by
\begin{equation} \label{eq:1-by-1-corres}
\pi_{(\mu, \eta_{t})}=\left(\mathrm{id}_{\R^2} \times \left(y \mapsto y+ t e(y) \right)\right)_{\#}\pi_{(\mu, \nu)} \qquad (\pi_{(\mu, \nu)} \in C(\mu,\nu), \, \pi_{(\mu, \eta_{t})} \in C(\mu, \eta_t)).
\end{equation}
Therefore,
\begin{equation}\label{eq:d-w1-mu-eta-t}
\begin{split}
d_{W_1}(\mu,\eta_t)
&=\inf \left\{ \iint_{\R^2 \times \R^2} \dm(x,z) \dd \pi_{(\mu, \eta_{t})} (x,z) \, \middle| \, \pi_{(\mu, \eta_{t})} \in C(\mu, \eta_t) \right\}\\
&=\inf \left\{ \iint_{\R^2 \times \R^2} \dm(x,y+t e(y)) \dd \pi_{(\mu, \nu)} (x,y) \, \middle| \, \pi_{(\mu, \nu)} \in C(\mu, \nu) \right\}\\
&=\inf \left\{ \iint_{\R^2 \times \R^2} (\dm(x,y)+t) \dd \pi_{(\mu, \nu)} (x,y) \, \middle| \, \pi_{(\mu, \nu)} \in C(\mu, \nu) \right\}\\
&=\inf \left\{ \iint_{\R^2 \times \R^2} \dm(x,y) \dd \pi_{(\mu, \nu)} (x,y) \, \middle| \, \pi_{(\mu, \nu)} \in C(\mu, \nu) \right\}+t\\
&=d_{W_1}(\mu, \nu)+t=t_0+t
\end{split}
\end{equation}
for every $t \geq 0.$

Note, that in the above computation, we heavily relied on the identity \eqref{eq:crucial-0}. The reversed triangle inequality implies that 
$$
d_{W_1}(\nu, \eta_t) \geq \left| d_{W_1}(\mu, \eta_t)- d_{W_1}(\mu, \nu)\right|=(t_0+t)-t_0=t.
$$
On the other hand, the cost of the coupling $\left(y \mapsto (y, y+ t e(y))\right)_{\#} \nu \in C(\nu, \eta_t)$ is simply $t,$ and hence $d_{W_1}(\nu, \eta_t)=t.$ Therefore, with the particular choice $t:=t_0=d_{W_1}(\mu,\nu)$ the triple $(\mu, \nu, \eta_{t_0})$ satisfies the requirement 
\begin{equation} \label{eq:req-v2}
d_{W_1}(\mu, \nu)=d_{W_1}(\nu, \eta_{t_0})=\frac{1}{2}d_{W_1}(\mu,\eta_{t_0})
\end{equation}
as every expression in \eqref{eq:req-v2} is equal to $t_0.$
\par
We turn to the proof of the direction (ii) $\Longrightarrow$ (i). The assumption (ii) implies in particular that for every $y \in \R^2$ there exists an $\eta \in \wort$ such that
\begin{equation} \label{eq:char-eq-part}
d_{W_1}(\mu, \delta_y)=d_{W_1}(\delta_y, \eta)=\frac{1}{2}d_{W_1}(\mu, \eta).
\end{equation}
Note that 
$$
d_{W_1}(\mu, \delta_y)=\int\limits_{\R^2} \dm(x,y) \dd \mu(x) \text{ and } d_{W_1}(\delta_y,\eta)=\int\limits_{\R^2} \dm(y,z) \dd \eta(z).
$$
Moreover, let $\pi_{(\mu,\eta)}^{*}$ denote an optimal coupling of $\mu$ and $\eta,$ and let us note that we have the following chain of inequalities:
\begin{equation} \label{eq:chain}
\begin{split}
d_{W_1}(\mu, \eta)
&=\iint_{\R^2 \times \R^2} \dm(x,z) \dd \pi_{(\mu,\eta)}^{*}(x,z)\\
&\leq \iint_{\R^2 \times \R^2} \dm(x,z) \dd (\mu \otimes \eta)(x,z)\\
&\leq \iint_{\R^2 \times \R^2} (\dm(x,y)+\dm(y,z)) \dd \mu(x) \dd \eta(z)\\
&=d_{W_1}(\mu, \delta_y)+ d_{W_1}(\delta_y,\eta).
\end{split}
\end{equation}
Therefore, \eqref{eq:char-eq-part} implies that both inequalities of \eqref{eq:chain} are saturated. The saturation of the first inequality means that $\mu \otimes \eta$ is an optimal coupling of $\mu$ and $\nu$ with respect to the transport cost $c(x,y)=\dm(x,y),$ while the saturation of the second inequality means that
\begin{equation} \label{eq:tri-sat}
\dm(x,z)=\dm(x,y)+\dm(y,z) \text{ for } \mu \otimes \eta \text{-almost every } (x,z) \in \R^2 \times \R^2. 
\end{equation}
In order to get a contradiction, assume that $\mu$ is not supported on a diagonal line, and let $x$ and $x'$ be points of the support of $\mu$ that do not lie on a common diagonal line. Now let us choose $y$ to be $y:=\frac{1}{2}(x+x').$
With this choice we get
\begin{equation} \label{eq:tri-not-sat}
\dm(x,z)<\dm(x,y)+\dm(y,z) \text{ or } \dm(x',z)<\dm(x',y)+\dm(y,z) \text{ for all } z \in \R^2 \setminus\{y\}.
\end{equation}
Indeed, it is easy to check --- see also Figure \ref{fig:tri-not-sat} --- that if both triangle inequalities in \eqref{eq:tri-not-sat} are saturated, then $z=y$ by necessity.
\begin{figure}[H]
\centering

\includegraphics[width=0.65\textwidth]{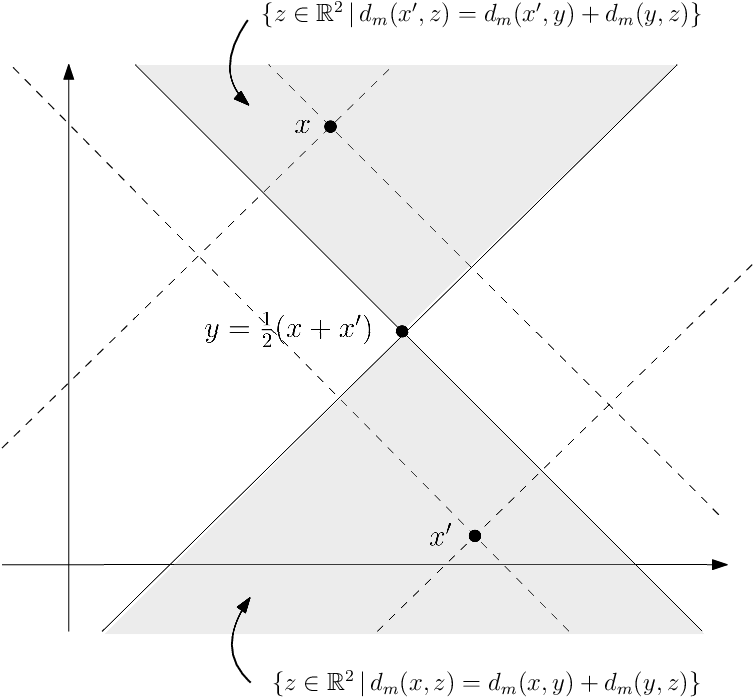}
\caption{Illustration for eq. \eqref{eq:tri-not-sat}}
\label{fig:tri-not-sat}
\end{figure}
Consequently, \eqref{eq:tri-sat} forces $\eta$ to be $\eta=\delta_y.$ But then $d_{W_1}(\delta_y,\eta)=0,$ which contradicts to \eqref{eq:char-eq-part}, because $d_{W_1}(\mu, \delta_y)>0$ as $\mu$ is not diagonally supported and hence not a Dirac. This contradiction completes the proof of the implication (ii) $\Longrightarrow$ (i).
\end{proof}

Now we give a metric characterization of the property that two measures $\mu_1$ and $\mu_2$ are supported on the {\it same}  diagonal line.

\begin{proposition} \label{prop:same-diag-line-supp-char}
Let $\mu_1, \mu_2 \in \wort.$ The following statements are equivalent.
\begin{enumerate}
    \item[(i)] \label{crit:i} $\mu_1$ and $\mu_2$ are supported on the same diagonal line.
    \item[(ii)] \label{crit:ii} For every $\nu \in \wort$ there exists an $\eta \in \wort$ such that
    \begin{equation} \label{eq:char-eq-2}
    d_{W_1}(\mu_i, \eta)=d_{W_1}(\mu_i, \nu)+d_{W_1}(\nu, \eta) \text{ for } i=1,2 \text{ and } d_{W_1}(\nu, \eta)=1.
    \end{equation}    In words, this item means that there is a measure $\eta$ aligned with both $(\mu_1, \nu)$ and $(\mu_2, \nu)$.
\end{enumerate}
\end{proposition}

\begin{proof}
Let us start with the proof of the direction (i) $\Longrightarrow$ (ii). Assume that $\mu_1$ and $\mu_2$ are supported on the diagonal line $L=\left\{(x_1, x_2) \in \R^2 \, \middle| \, x_2= \varepsilon x_1 +a\right\}.$
Let us recall the allocation of directions \eqref{eq:direction-alloc} and its crucial property \eqref{eq:crucial-0}.
Let $\eta$ be defined by
\begin{equation} \label{eq:eta-def}
\eta:=\left(y \mapsto y + e(y) \right)_{\#} \nu .
\end{equation}
Note that \eqref{eq:eta-def} is a special case of \eqref{eq:eta-t-def} with $t=1,$ and hence $d_{W_1}(\mu_i, \eta)=d_{W_1}(\mu_i, \nu)+1$ for $i=1,2.$ 

Similarly as in the previous proposition, the reverse triangle inequality ensures that $d_{W_1}(\nu, \eta) \geq \left|d_{W_1}(\mu_i,\eta)-d_{W_1}(\mu_i,\nu)\right|=1,$ and the transport map $(y \mapsto (y,y+e(y)))_{\#}\nu$ between $\nu$ and $\eta$ shows that $d_{W_1}(\nu, \eta)=1$ which completes the proof of this direction.\\

To prove the direction (ii) $\Longrightarrow$ (i), note that by the previous statement, both of the measures $\mu_1$ and $\mu_2$ are supported on some diagonal line. Assume by contradiction that $\mu_1$ and $\mu_2$ are not supported on the same diagonal line, and hence in particular there exist points $x_1 \in \mathrm{supp}(\mu_1)$ and $x_2 \in \mathrm{supp}(\mu_2)$ that do not lie on a common diagonal line. As in the proof of Proposition \ref{prop:diag-line-supp-char} let us choose $\nu:=\delta_y$ where $y=\frac{1}{2}(x_1+x_2).$ With this choice, \eqref{eq:char-eq-2} implies that 
$$
\dm(x,z)=\dm(x,y)+\dm(y,z) \text{ for } \mu_i \otimes \eta \text{-almost every } (x,z) \in \R^2 \times \R^2 \qquad (i=1,2). 
$$
In particular, $\dm(x_i,z)=\dm(x_i,y)+\dm(y,z)$ for $i=1,2$ which implies $z=y$ for $\eta$-almost every $z,$ and hence forces $\eta$ to be $\delta_y.$ However, this contradicts the requirement $d_{W_1}(\nu,\eta)=1,$ so we got the desired contradiction.
\end{proof}

Now we are in the position to prove the main result of this section.

\begin{theorem}\label{T: main-1}
  The Wasserstein space $\wort$ is diagonally rigid. That is, for any isometry 
  $\Phi: \wort \to \wort$ there exists an isometry $T: \R^2 \to \R^2$ such that $\Phi\circ T_\#$ fixes all measures supported on $L_+$ and $L_-$.
\end{theorem}

\begin{proof}
Let $\Phi: \wort \to \wort$ be an isometry, and let $\mu,\mu' \in \mathcal{W}_1(L_{+})$ be two measures, $\mu\neq\mu'$. According to Proposition \ref{prop:same-diag-line-supp-char}, their images $\Phi(\mu)$ and $\Phi(\mu')$ are supported on a diagonal line $L_{\varepsilon,a}$ for a suitable $\varepsilon\in\{-1,1\}$ and $a\in\mathbb{R}$. Since for every $\varepsilon\in\{-1,1\}$ and $a\in\mathbb{R}$ there is an isometry $T: \R^2 \to \R^2 $ that maps $L_{\varepsilon, a} $ onto $L_+$, we can assume that $\supp\big(\Phi(\mu)\big)\subseteq L_+$ and $\supp\big(\Phi(\mu')\big)\subseteq L_+$. In fact, for every $\nu$ with $\supp(\nu)\subseteq L_+$ we conclude that $\supp\big(\Phi(\nu)\big)\subseteq L_+$. Indeed, let us repeat the above argument for $\mu$ and $\nu$. Since they are both supported on $L_+$ their images are supported on the same diagonal line. We already know that $\supp\big(\Phi(\mu)\big)\subseteq L_+$, and therefore if $\Phi(\mu)$ is not a Dirac mass, then Proposition \ref{prop:same-diag-line-supp-char} guarantees that $\supp\big(\Phi(\nu)\big)\subseteq L_+$. If $\Phi(\mu)$ is a Dirac mass, say $\Phi(\mu)=\delta_{(x,x)}$ then we have to exclude the possibility of $\supp\big(\Phi(\nu)\big)\subseteq L_{-1,2x}$. To this aim, consider $\nu$ and $\mu'$, and again, apply Proposition \ref{prop:same-diag-line-supp-char} to conclude that $\supp\big(\Phi(\mu')\big)\subseteq L_{-1,2x}$. But this is a contradiction, as we already know that $\supp\big(\Phi(\mu')\big)\subseteq L_+$ and therefore $\supp\big(\Phi(\mu')\big)=\{(x,x)\}$, or equivalently $\Phi(\mu')=\delta_{(x,x)}=\Phi(\mu)$.

We obtain in this way, that $\Phi$ restricted to $\mathcal{W}_1(L_{+})$ is a (bijective) isometry of $\mathcal{W}_1(L_{+}),$ which is isomorphic to $W_1(\R, |\cdot|)$ which is known to be isometrically rigid --- see \cite{GTV1}. Therefore measures in $\mathcal{W}_1(L_{+})$ are left invariant by $\Phi.$\\

Finally we need to show that measures in $\mathcal{W}_1(L_{-})$ are also left invariant by $\Phi.$
To see this note that $(0,0) \in L_{+} \cap L_{-}$ and considering the measure $\mu = \delta_{(0,0)}$ together with another measure $\nu$ supported on $L_{-}$ we conclude by applying Proposition  \ref{prop:same-diag-line-supp-char} that both $\Phi(\mu)$ and $\Phi(\nu)$ are supported on the same diagonal line. Since we know already that $\Phi(\mu) = \mu = \delta_{(0,0)}$ we can conclude that the support of  $\Phi(\nu)$ is in $L_+$ or it is in $L_-$. Since the first option cannot hold as the pre-images of measures supported on $L_+$ are supported on $L_+$ by the previous paragraph, we are left with the second one. This shows that if $ \nu \in \mathcal{W}_1(L_{-})$ then so is $\Phi(\nu)$. By possibly applying another isometry of $\R^2$ we obtain that $\Phi$ fixes the elements of $ \nu \in \mathcal{W}_1(L_{-})$ as well. 
\end{proof}

\subsection{Diagonal rigidity of $\wprt$ for $p>1$}
In this subsection, we show that the Wasserstein space $\wprt$ for $p>1$ is diagonally rigid. First, we give a metric characterization of Dirac measures. Such a characterization will guarantee that if $\mu$ is a Dirac mass, then $\Phi(\mu)$ is a Dirac mass as well.

\begin{proposition} \label{prop:dirac-char-p>1}
Let $p>1$ and $\mu \in \wprt.$ The following statements are equivalent.
\begin{enumerate}
    \item[(i')] \label{crit:I} $\mu$ is a Dirac mass, that is, $\mu=\delta_x$ for some $x \in \R^2.$
    \item[(ii')] \label{crit:II} For every $\nu \in \wprt$ there exists an $\eta \in \wprt$ such that
    \begin{equation}
    d_{W_p}(\mu, \nu)=d_{W_p}(\nu, \eta)=\frac{1}{2}d_{W_p}(\mu, \eta).
    \end{equation}
\end{enumerate}
\end{proposition}
  In words, item (ii') means that $\mu$ admits a
symmetrical with respect to every other measure.

Note that the above Proposition \ref{prop:dirac-char-p>1} characterizing \emph{Dirac masses} in $\wprt$ for $p>1$ highlights the difference between the cases $p=1$ and $p>1$. This statement is very similar in spirit to Proposition \ref{prop:diag-line-supp-char} characterizing \emph{measures supported on diagonal lines} in $\wort.$ In fact, condition (ii) of Proposition \ref{prop:diag-line-supp-char} is the same as condition (ii') of Proposition \ref{prop:dirac-char-p>1}, up to a modification in the parameter value of the Wasserstein distance that we consider. This means that diagonally supported measures play the role of Dirac masses in the case $p=1,$ and in particular, there are plenty of examples of {\it non-Dirac measures} satisfying condition (ii), which is the $1$-Wasserstein version of condition (ii') above. These examples are explicitly constructed in the proof of the (i) $\Longrightarrow$ (ii) part of Proposition \ref{prop:diag-line-supp-char}.

\begin{proof}
    
Let us prove the direction (i') $\Longrightarrow$ (ii') first. Let $x \in \R,$ let $\mu=\delta_x,$ and let us define the following dilation with center $x$ on $\R^2$:
\begin{equation} \label{eq:dilation-def}
D_x: \R^2 \to \R^2; \, y \mapsto D_x(y):=x+2(y-x).
\end{equation}
Now, for any $\nu \in \wprt,$ let us define the corresponding $\eta_{\nu}$ by
\begin{equation} \label{eq:eta-nu-def}
\eta_{\nu}:=\left(D_x\right)_{\#}\nu.
\end{equation}
It is clear that $d_{W_p}(\mu, \eta_{\nu})=2 d_{W_p}(\mu, \nu).$ Indeed,
\begin{equation*}
\begin{split}
    d_{W_p}(\delta_x, \eta_{\nu})
&=\left(\int\limits_{\R^2}\dm^p(x,z)\dd(D_x)_{\#}\nu (z)\right)^{\frac{1}{p}}\\
&=\left(\int\limits_{\R^2}\dm^p(x,x+2(y-x))\dd\nu (y)\right)^{\frac{1}{p}}\\
&=\left(\int\limits_{\R^2}2^p \dm^p(x,y))\dd\nu (y)\right)^{\frac{1}{p}}\\
&=2 d_{W_p}(\delta_x, \nu).
\end{split}
\end{equation*}

Moreover, by the reversed triangle inequality, $d_{W_p}(\nu, \eta_{\nu}) \geq d_{W_p}(\mu, \eta_{\nu})-d_{W_p}(\mu, \nu)=d_{W_p}(\mu, \nu),$ while the obvious coupling of $\nu$ and $\eta_{\nu}$ given by the dilation $D_x$ guarantees that $d_{W_p}(\nu, \eta_{\nu}) \leq d_{W_p}(\mu, \nu),$ and hence the direction (i') $\Longrightarrow$ (ii') is proved.
\par
To prove the other direction (ii') $\Longrightarrow$ (i'), let $\nu$ be a Dirac mass, $\nu=\delta_y,$ and
let $\eta$ be as in condition (ii'). Then

\begin{equation*}
\begin{split}
d_{W_p}(\mu, \eta)&=\left(\iint_{\R^2 \times \R^2} \dm^p(x,z) \dd \pi_{(\mu,\eta)}^{*}(x,z)\right)^{\frac{1}{p}}\\
&\leq \left(\iint_{\R^2 \times \R^2} \dm^p(x,z) \dd (\mu \otimes \eta)(x,z)\right)^{\frac{1}{p}}\\
&\leq \left(\iint_{\R^2 \times \R^2} (\dm(x,y)+\dm(y,z))^p \dd (\mu \otimes \eta)(x,z)\right)^{\frac{1}{p}}\\
&\leq  \left(\iint_{\R^2 \times \R^2} \dm(x,y)^p \dd (\mu \otimes \eta)(x,z)\right)^{\frac{1}{p}}
+  \left(\iint_{\R^2 \times \R^2} \dm^p(y,z) \dd (\mu \otimes \eta)(x,z)\right)^{\frac{1}{p}}\\
&=  \left(\int\limits_{\R^2} \dm(x,y)^p \dd \mu(x)\right)^{\frac{1}{p}}
+  \left(\int\limits_{\R^2} \dm^p(y,z) \dd \eta(z)\right)^{\frac{1}{p}}\\
&=\left(\iint_{\R^2 \times \R^2} \dm(x,z)^p \dd (\mu \otimes \delta_y)(x,z)\right)^{\frac{1}{p}}
+  \left(\iint_{\R^2 \times \R^2} \dm^p(x,z) \dd (\delta_y \otimes \eta)(x,z)\right)^{\frac{1}{p}}\\
&=d_{W_p}(\mu,\nu)+d_{W_p}(\nu,\eta).
\end{split}
\end{equation*}
Since we assumed that $d_{W_p}(\mu, \nu)=d_{W_p}(\nu, \eta)=\frac{1}{2}d_{W_p}(\mu, \eta)$, all the inequalities in the above chain are saturated. In particular, by the saturation of the $L_p$-Minkowski inequality for $p>1$ (strictly convex norm), we get that there is a nonnegative constant $\alpha \geq 0$ such that 
\begin{equation}
\label{eq:mink-sat}
\dm(y,z)=\alpha \dm(x,y) \qquad \text{ for } \mu\text{-a.e. } x \in \R^2 \text{ and for } \eta\text{-a.e. } z \in \R^2.
\end{equation}
If $\mu$ is not a Dirac mass, then let $x_1$ and $x_2$ be two different points in its support, and let $y:=\frac23 x_1 +\frac13 x_2.$ Then the left hand side of \eqref{eq:mink-sat} is independent of $x,$ while the right hand side is not -- a contradiction.

\par
\end{proof}

The next step is to find a metric characterization of the property that the support of a measure $\mu$ is diagonally aligned with a point $x$ in the underlying space, that is, $\mathrm{supp}(\mu) \subset (x+ L_{+}) \cup (x + L_{-}).$ This metric characterisation turns out to be the property that there is only one $p$-Wasserstein geodesic between $\delta_x$ measure and $\mu$.

\begin{proposition}\label{prop:Dirac-and-mu} Let $p>1$ and $x\in\mathbb{R}^2$ be fixed. For a measure $\mu\in\wprt$ the following statements are equivalent:
\begin{itemize}
    \item[(a)] There exists a unique unit-speed geodesic segment $\gamma$ such that $\gamma(0)=\delta_x$ and $\gamma(T)=\mu$, where $d_{W_p}(\delta_x,\mu)=T$.
    \item[(b)] We have the inclusion $\mathrm{supp}\mu\subseteq D_x$, where 
    $D_x = (x + L_+) \cup (x+L_-)$.
\end{itemize}
\end{proposition}

\begin{proof}
    
 To prove the statement let us note that Corollary 7.22 in Villani's book \cite{Villani} says that if $p>1$, and the underlying metric space is a complete, separable, and locally compact length space, then constant-speed geodesics connecting two measures are all displacement interpolations, i.e., geodesics are always constructed from optimal transport plans.

Therefore (see Corollary 7.23 in \cite{Villani}), if we want to guarantee that there is only one geodesic between two measures $\mu,\nu$, we need two properties:
\begin{itemize}
    \item[$\bullet$] we need a unique optimal transport plan $\widetilde{\pi}$ 
    \item[$\bullet$] and for $\widetilde{\pi}$-almost every $(x,y)$, $x$ and $y$ must be joined by a unique geodesic.
\end{itemize}
Let us note that the first property is automatically satisfied since one of the masses that we consider is a Dirac mass. Furthermore, note that in  $(\mathbb{R}^2,\dm)$, the second property means exactly that $x$ and $y$ are diagonally aligned, i.e. both points lie on the same diagonal line. For a fixed $x\in\mathbb{R}^2$ let us denote by $D_x$ the set of those points that are diagonally aligned with $x$. Of course, $D_x$ is the union of the two diagonal lines $L_++x$ and $L_{-}+x$ passing through $x$ concluding the proof. 
\end{proof}
Now we are in position to prove that $\wprt$ is diagonally rigid for $p>1$.

\begin{theorem} \label{T: main-2}
For all $p>1$ the Wasserstein space $\wprt$ is diagonally rigid.
\end{theorem}
\begin{proof}
Let $\Phi$ be an isometry. Since Proposition \ref{prop:dirac-char-p>1} is a metric characterization of Dirac masses, we know that $\Phi$ maps the set of Dirac masses onto itself. That is, there exists a bijection $T:\mathbb{R}^2\to\mathbb{R}^2$ such that $\Phi(\delta_x)=\delta_{T(x)}$. In fact, $T$ is an isometry, as $d_{W_p}(\delta_x,\delta_y)=\dm(x,y)$ for all $x,y\in\mathbb{R}^2$. Without loss of generality, we can assume that $T(x)=x$, and thus $\Phi(\delta_x)=\delta_x$ for all $x\in\mathbb{R}^2$. Next, consider the diagonal line $L_+$. (The case of $L_-$ is completely analogous.) Fix an arbitrary $x\in L_+$, and observe that according to Proposition \ref{prop:Dirac-and-mu}, for any $\mu$ such that $\mathrm{supp}(\mu)\subseteq L$, there is only one geodesic connecting $\delta_x$ and $\mu$. Therefore, there must be only one geodesic between $\Phi(\delta_x)=\delta_x$ and $\Phi(\mu)$. Again, according to Proposition \ref{prop:Dirac-and-mu} this means that $\mathrm{supp}\big(\Phi(\mu)\big)\subseteq D_x$. Now choose a $y\in L_+$ ($y\neq x$) and repeat the argument. The conclusion is that $\mathrm{supp}\big(\Phi(\mu)\big)\subseteq D_y$. But $D_x\cap D_y=L_+$, and therefore $\mathrm{supp}\big(\Phi(\mu)\big)\subseteq L_+$. Now we know that $\Phi$ sends measures supported on $L_+$ into measures supported on $L_+$. Since $L_+$ endowed with the restriction of $\dm:\mathbb{R}^2\times\mathbb{R}^2\to\mathbb{R}_+$ onto $L_+\times L_+$ is nothing else but $(\mathbb{R},|\cdot|)$, the set
$$\mathcal{W}_p(L_+,\dm):=\{\mu\in\wprt\,|\,\mathrm{supp}\mu\subseteq L_+\}$$ endowed with the Wasserstein distance
is isometrically isomorphic to the Wasserstein space $\mathcal{W}_p(\mathbb{R},d_{|\cdot|})$ investigated in \cite{GTV1,K}, which is isometrically rigid if $p\neq 2$. Since we assumed that $\Phi(\delta_x)=\delta_x$, isometric rigidity forces the restriction $\Phi|_{\mathcal{W}_p(L_+,\dm)}$ to be the identity, i.e. $\Phi(\mu)=\mu$ for all $\mu$ supported on $L_+$. The same argument with $L_-$ completes the proof in the $p\neq2$ case.\\

If $p=2$ we need to use a more involved argument, since $\mathcal{W}_2(\mathbb{R},d_{|\cdot|})$ is not isometrically rigid. In fact, if $\Psi:\mathcal{W}_2(\mathbb{R},d_{|\cdot|})\to\mathcal{W}_2(\mathbb{R},d_{|\cdot|})$ is an isometry, then $\Psi(\delta_x)=\delta_x$ for all $x\in \mathbb{R}$ itself does not imply $\Psi(\mu)=\mu$ for all $\mu\in\mathcal{W}_2(\mathbb{R},d_{|\cdot|})$, as $\mathcal{W}_2(\mathbb{R},d_{|\cdot|})$ admits exotic isometries and a nontrivial shape-preserving isometry. 

Therefore, even if we know that $\Phi(\delta_{(x,x)})=\delta_{(x,x)}$ for all $(x,x)\in L_+$ we cannot guarantee yet that $\Phi(\mu)=\mu$ for all $\mu$ with $\supp(\mu)\subseteq L_+$. We have to rule out that the restriction of $\Phi$ onto $\mathcal{W}_p(L_+,\dm)$ does not act like a non-trivial isometry. This boils down to investigating the action on measures whose support consists of two points of $L_+$ as follows: using the isometric identification $t\mapsto(t,t)$ between the real line and $L_+$, we will use Kloeckner's result which tells us how the image of a $\mu=\alpha\delta_{(x,x)}+(1-\alpha)\delta_{(y,y)}$ would look like if $\Phi$ would act on $L_+$ like non-trivial isometry. Then we will choose a special $\mu$ and a special Dirac measure $\delta_{(u,v)}$ (not supported on $L_+$) with the property that $d_{W_2}(\mu,\delta_{(u,v)})\neq d_{W_2}(\Phi(\mu),\delta_{(u,v)})=d_{W_2}(\Phi(\mu),\Phi(\delta_{(u,v)}))$, a contradiction. 

Let us introduce some notations. For the diagonal line $L_+$ the set of measures supported on two points of $L_+$ will be denoted by $\Delta_2$
\begin{equation}\label{delta2L}
    \Delta_2=\{\alpha\delta_{(x,x)}+(1-\alpha)\delta_{(y,y)}\,|\, \alpha\in(0,1),\, x,y\in \mathbb{R}\}.
\end{equation}
Following the notations in  Kloeckner's paper \cite{K}, elements of $\Delta_2$ will be parametrized by three parameters $m\in\mathbb{R}$, $\sigma\geq0$, and $r\in\mathbb{R}$ as follows:
\begin{equation}\label{parametrization}
\mu(m,\sigma,r)=\frac{e^{-r}}{e^r+e^{-r}}\delta_{(m-\sigma e^r,m-\sigma e^r)}+\frac{e^{r}}{e^r+e^{-r}}\delta_{(m+\sigma e^{-r},m+\sigma e^{-r})}.
\end{equation}
According to Lemma 5.2 in \cite{K}, if an isometry $\Phi$ fixes all Dirac masses, then its action on $\Delta_2$ is
$$\Phi\big(\mu(m,\sigma,r)\big):=\mu(m,\sigma,\varphi(r))$$
where $\varphi:\R\to\R$ is an isometry. In other words, $\Phi$ is equal to the shape-preserving isometry
\begin{equation}\label{eq:Phistar}
\Phi^*:\mathcal{W}_2(L_+)\to\mathcal{W}_2(L_+),\quad\Phi^*\big(\mu(m,\sigma,r)\big):=\mu(m,\sigma,-r),
\end{equation} 
or $\Phi$ is equal to an exotic isometry
\begin{equation}\label{eq:Phit}
\Phi_t:\mathcal{W}_2(L_+)\to\mathcal{W}_2(L_+),\quad\Phi^t\big(\mu(m,\sigma,r)\big):=\mu(m,\sigma,r+t)
\end{equation}
for some $t\neq0$, or $\Phi$ is the composition $\Phi_t\circ\Phi^*$ for some $t\neq0$. Note, that if $t=0$, then $\Phi_t$ is the identity, so $\Phi_0$ is not an exotic isometry.\\

To handle the case $\Phi^*$, choose $\mu=\mu(0,1,\ln2)=\frac15\delta_{(-2,-2)}+\frac45\delta_{(\frac12,\frac12)}$. Then $$\Phi^*(\mu)=\mu(0,1,-\ln2)=\frac45\delta_{(-\frac12,-\frac12)}+\frac15\delta_{(2,2)}.$$ Calculating the Wasserstein distance of $\mu$ and $\Phi(\mu)$ from $\delta_{(2,0)}$, we obtain that
$$d_{W_2}\big(\delta_{(2,0)},\mu(0,1,\ln2)\big)=\sqrt{5}\quad\mbox{and}\quad d_{W_2}\big(\delta_{(2,0)},\mu(0,1,-\ln2)\big)=\sqrt{5+\frac45},$$
So if $\Phi$ is an isometry, it cannot act on $L_+$ like $\Phi^*$.\\

To handle the case $\Phi_t$ for $t>0$, choose $\mu=\mu(0,1,0)=\frac12\delta_{(-1,-1)}+\frac12\delta_{(1,1)}$. Then $\Phi(\mu(0,1,0))=\mu(0,1,t)$. Now fix the Dirac measure $\delta_{(-1,0)}$ and calculate both $d_{W_2}(\delta_{(-1,0)},\mu(0,1,0))$ and $d_{W_2}(\delta_{(-1,0)},\mu(0,1,t))$. Again, if $\Phi$ would act like $\Phi_t$ on $L_+$, we should get the same result, since $\Phi$ is an isometry. The calculation shows that
$d_{W_2}\big(\delta_{(-1,0)},\mu(0,1,0)\big)=\sqrt{\frac52}$ and 
$$d_{W_2}\big(\delta_{(-1,0)},\mu(0,1,t)\big)=\sqrt{2+\frac{2-e^{-t}}{e^t+e^{-t}}}.$$

These two numbers are equal if and only if $t=0$ or $t=\ln3$. We assumed that $t>0$, so one single Dirac $\delta_{(-1,0)}$ excluded all $\Phi_t$ except $t=\ln3$. Choosing a different Dirac measure, say $\delta_{(-\frac12,0)}$, a simple calculation shows that 
$$\sqrt{\frac{13}{8}}=d_{W_2}\big(\delta_{(-\frac12,0)},\mu(0,1,0)\big)=d_{W_2}\big(\delta_{(-\frac12,0)},\mu(0,1,\ln3)\big)=\sqrt{\frac{12}{8}+\frac{1}{40}},$$
a contradiction. Similar calculations for $\Phi_t$ with negative $t$ and for $\Phi^*\circ\Phi_t$ show that the only case when we don't get a contradiction is when $\Phi$ acts as $\Phi_0$, which is the identity.
\end{proof}

\subsection{Diagonal rigidity of $\woq$}
In this subsection, we consider the case when $X= Q= [-1,1]^2$ and $p=1$. Diagonal rigidity is achieved as a result of the following statements about measures that are supported on the sides, at the corners, and finally on the diagonals of $Q$. The first statement concerns measures supported on the opposite sides of the boundary of $Q $ and it is valid for all $p\geq 1$.

\begin{lemma} \label{side-measures} 
Let $p\geq 1$ and $\Phi: \wpq \to \wpq$ be an isometry. If   $\mu, \nu \in \wpq $ are two probability measures whose supports lie on opposite sides of the closed unit ball $Q= [-1,1]^2$, then their isometric images $\Phi(\mu)$ and $\Phi(\nu)$ have the same property.
\end{lemma}

\begin{proof}
Let us note that if $x,y \in Q$ are any two points then $\dm(x,y) \leq 2 $ with equality if and only if $x$ and $y$ lie on two opposite sides of $Q$. 

This implies that if $\mu, \nu \in \wpq$, then 
\begin{equation} \label{sides}
d_{W_p}(\mu, \nu) \leq 2,
\end{equation} 
with equality if and only if the supports $\supp(\mu)$ and $\supp(\nu) $ are contained in two opposite sides of $Q$. 
To see this, note that inequality \eqref{sides} follows immediately from the definition of the Wasserstein metric $d_{W_p}$ and the fact that $\dm(x,y) \leq 2$ for all $x,y\in Q$. Furthermore, if $\supp(\mu)$ and $\supp(\nu) $ are contained in two opposite sides of $Q$, then $\dm(x,y) = 2$ for all $x\in\supp (\mu)$ and $y\in \supp(\nu)$.

Let $\pi_0 $ be an optimal coupling of $\mu$ and $\nu$. Since $\supp(\pi_0) \subseteq\supp (\mu) \times\supp (\nu)$ we have that $\dm(x,y)= 2$ for any $(x,y) \in\supp (\pi_0)$ and therefore
$$ d_{W_p}^p(\mu, \nu) = \int\limits_{Q\times Q} \dm^p(x,y) ~\mathrm{d}\pi_0(x,y) = 2^p.$$
To show the converse, let us assume that  $d_{W_p}^p(\mu, \nu)  = 2^p$. Then for all couplings 
$\pi \in C(\mu, \nu)$ we have 
$$2^p =d_{W_p}^p(\mu, \nu)= \int\limits_{Q\times Q} \dm^p(x,y) ~\mathrm{d}\pi(x,y) \leq  \int\limits_{Q\times Q} 2^p  ~\mathrm{d}\pi(x,y) = 2^p,$$
thus $\dm(x,y) =2 $ for $\pi$ almost all $(x,y)$. This applies for $\pi = \mu\otimes \nu$, and so $\mu$ and $\nu$ must be concentrated on the opposite sides of $Q$.

The statement of the lemma is now an immediate consequence of this claim. Indeed,  assume that $\supp(\mu)$ and $\supp(\nu) $ are contained in two opposite sides of $Q$. Then we have $d_{W_p}^p(\mu, \nu)=2^p$. Since  $\Phi: \wpq \to \wpq$ is an isometry, we have $d_{W_p}^p(\Phi(\mu), \Phi(\nu))=2^p.$ But, then the supports of the two measures $\Phi(\mu)$ and $\Phi(\nu)$ must be contained in two opposite sides of $Q$. 
\end{proof}
\begin{corollary} \label{corners} 
Let us fix a $p\geq1$ and denote by $ D= \{ x_1, x_2, x_3, x_4 \} $ the set of vertices of $Q$ and let $V$ be the set of Dirac measures supported on the points of $D$, i.e., 
$$V=\{ \delta _{x_1} = \delta_{(-1,-1)} , \delta_{x_2} =\delta_{(1,-1)} , \delta_{x_3}=\delta_{(1,1)}, \delta_{x_4} = \delta_{(-1,1)} \}.$$  Given any isometry  $\Phi: \wpq \to \wpq$, there exists an isometry $\Psi: (Q, \dm) \to (Q, \dm)$ such that $\Phi \circ \Psi_{\#} (\delta_x) = \delta_x$  for all $x\in D$.
\end{corollary}

\begin{proof} Let us denote by $S_1, S_3$ the two vertical and by $S_2, S_4$ the two horizontal sides of $Q$ such that $S_1$ is the left vertical and $S_4$ is the top horizontal side. 

We apply Lemma \ref{side-measures} to every pair of elements of $V,$ and we conclude that for every $j \in \{1,2,3,4\},$ the measure $\Phi(\delta_{x_j})$ is supported on a certain side $S_i$ of $Q.$ 

We claim that each $\Phi(\delta_{x_j})$ must be in fact supported on some vertex. To see this we argue by contradiction. Let us assume that for example  $\Phi(\delta_{x_1})$ is supported on one of the sides, say $S_1$ but not on any of the vertices of $S_1$. 

It is clear by Lemma \ref{side-measures} that all other measures $\Phi(\delta_{x_i})$ for $i=2,3,4$ must be supported on the opposite side of $S_1$, that is $S_3.$ But the mutual distance of any pair of these measures must be equal to $2$ which shows that any two of these three measures must be lying on opposite sides again, which is $S_2$ and $S_4$. But there are only two possibilities for measures with support in $S_2\cap S_3$ and $S_4 \cap S_3$, namely the two Dirac masses on the vertices of $S_3$ which gives a contradiction. 

It is easy to see that there exists an isometry $\Psi: Q \to Q$ such that 
$\Phi \circ \Psi_{\#}$ fixes $\delta_{x_i}$ and $\delta_{x_{i+1}}$ for some $i \in \{1,2,3,4\},$ and therefore we can assume without loss of generality that $(\Phi \circ \Psi_{\#})(\delta_{x_i})=\delta_{x_i}$ for $i=1,2.$ This will imply that all measures supported on $S_4$ are fixed. We must show that $\Phi \circ \Psi_{\#} (\delta_{x_i}) = \delta_{x_i}$ for $i=3,4$. Note, that the map 
$$\Phi \circ \Psi_{\#}: \wpq \to \wpq$$ is itself an isometry. 
Let us assume indirectly that $\Phi \circ \Psi_{\#} (\delta_{x_3}) = \delta_{x_4}$ (which implies that and $\Phi \circ \Psi_{\#} (\delta_{x_4}) = \delta_{x_3}$), and take a $\xi$ such that $\supp(\xi)\subseteq S_3$, $\xi\notin\{\delta_{x_2},\delta_{x_3}\}$. Since $d_{W_p}(\delta_{x_1},\xi)=2$, we have 
$$
2=d_{W_p}(\delta_{x_1},\xi)=d_{W_p}(\Phi(\delta_{x_1}),\Phi(\xi))=d_{W_p}(\delta_{x_1},\Phi(\xi)),
$$
which implies that $\supp(\Phi(\xi))\subseteq S_3\cup S_4$.  Similarly, $d_{W_p}(\delta_{x_4},\xi)=2$, and thus 
$$
2=d_{W_p}(\delta_{x_4},\xi)=d_{W_p}(\Phi(\delta_{x_4}),\Phi(\xi))=d_{W_p}(\delta_{x_3},\Phi(\xi)),
$$ 
which implies that $\supp(\Phi(\xi))\subseteq S_1\cup S_2$. Combining $\supp(\Phi(\xi))\subseteq S_1\cup S_2$ and $\supp(\Phi(\xi))\subseteq S_3\cup S_4$ with $\xi\notin\{\delta_{x_2},\delta_{x_3}\}$, we get that $\Phi(\xi)=\alpha^{\ast}\delta_{x_2}+(1-\alpha^{\ast}) \delta_{x_4}$ for some $\alpha^{\ast}\in(0,1).$ (Recall that $\Phi(\delta_{x_2})=\delta_{x_2}$, $\Phi(\delta_{x_3})=\delta_{x_4}$, and $\Phi$ is injective.) If $p>1$, then this is a contradiction. Indeed, choose $\xi:=\delta_y$ with $y\in S_3\setminus\{x_2,x_3\}$, and observe that the triple $\delta_{x_2}$, $\delta_y$ and $\delta_{x_3}$ saturates the triangle inequality, but the triple $\Phi(\delta_{x_2})=\delta_{x_2}$, $\Phi(\delta_y)=\alpha^{\ast}\delta_{x_2}+(1-\alpha^{\ast}) \delta_{x_4}$ and $\Phi(\delta_{x_3})=\delta_{x_4}$ does not, as $\sqrt[p]{2^p(1-\alpha^{\ast})}+\sqrt[p]{2^p\alpha^{\ast}}\neq2$. If $p=1$, we need a different argument. The set $\mathcal{I}:=\{ \mu_{\alpha}:=\alpha\delta_{x_2}+(1-\alpha) \delta_{x_4}\,|\,\alpha\in(0,1)\}$ is isometric to the set $\big((0,1),2|\cdot|\big)$, since $d_{W_1}(\mu_{\alpha},\mu_{\beta})=|\alpha-\beta|$. Therefore for any three different elements $\mu_{\alpha_1}, \mu_{\alpha_3}, \mu_{\alpha_3}$ there exists a bijection $\sigma:\{1,2,3\}\to\{1,2,3\}$ such that $d_{W_1}(\mu_{\alpha_{\sigma(1)}},\mu_{\alpha_{\sigma(3)}})=(\mu_{\alpha_{\sigma(1)}},\mu_{\alpha_{\sigma(2)}})+(\mu_{\alpha_{\sigma(2)}},\mu_{\alpha_{\sigma(3)}})$.
Now choose a non-degenerate triangle $\xi_1,\xi_2,\xi_3$ supported on $S_3\setminus\{x_2,x_3\}$ in the sense that they do not saturate the triangle inequality in any order. The existence of such a triple is a contradiction, as an appropriate permutation of their image in $\mathcal{I}$ will saturate the triangle inequality.

\end{proof}

 From now on we shall assume without loss of generality that our isometry 
 $$\Phi: \wpq \to \wpq$$
 fixes the elements of $V$, i.e. the Dirac masses on the corners of $Q$. Note that this property implies by Lemma \ref{side-measures} that any measure supported on one of the sides of $Q$ will be mapped to a measure supported on the same side of $Q$.

In what follows we shall prove, even a stronger property for measures supported on the main diagonals 
 $$ L_+ = \{ (t,t): t\in [-1,1] \}, \ \textrm{and } \ L_- = \{ (t,-t): t\in [-1,1] \},$$
 namely, that they are fixed under the action of the isometry. This is valid for the case $p=1$.

\begin{theorem} \label{T: main-3}
The Wasserstein space $\woq$ is diagonally rigid.
\end{theorem}  
\begin{proof} By Corollary \ref{corners}, we can assume without loss of generality that $\Phi$ fixes the Dirac masses at the four corners of $Q$. It is enough to show that if  $\mu \in \woq$ supported on  $L_+ $, then $\Phi(\mu)=\mu$ as the case of $L_{-}$ is similar.
We show first that if $\supp(\mu) \subseteq L_+$, then $\supp\big(\Phi(\mu)\big) \subseteq L_+$.
This is based on the following observation: If $x \in Q $ then 
\begin{equation} \label{diagonal-equ}
\dm((-1,-1), x ) + \dm(x, (1,1))\geq \dm((-1,-1), (1,1)) = 2 
\end{equation}
with equality if and only if $x \in L_+$.

The above inequality follows simply by the triangle inequality applied for the metric $\dm$. The characterization of the equality case is slightly more tricky. It is based on the fact that the line segment $ L_+: t\to (t,t), t\in [-1,1] $ is the only geodesic with respect to the metric $\dm$ connecting the endpoints $(-1,-1)$ and $(1,1)$. 
This observation has the following consequence for measures: If $\mu \in \woq$, then 
\begin{equation} \label{diagonal-ineq-measu}
d_{W_1}(\delta_{(-1,-1)}, \mu) + d_{W_1}(\mu, \delta_{(1,1)}) \geq 2,
\end{equation}
with equality if and only if $\supp(\mu) \subseteq L_+$. 
To show inequality \eqref{diagonal-ineq-measu} we integrate inequality  \eqref{diagonal-equ} with respect to $\mu$. In this way we obtain
\begin{align} \label{diagonal-int}
\begin{split}
   d_{W_1}(\delta_{(-1,-1)}, \mu) + d_{W_1}(\mu, \delta_{(1,1)})&= \int\limits_{Q}\dm((-1,-1), x )~\mathrm{d}\mu(x) + \int\limits_{Q} \dm(x, (1,1))  ~\mathrm{d}\mu(x)\\   &\geq \int\limits_Q 2~\mathrm{d}\mu(x)  = 2 .
\end{split}
\end{align}
Let us assume that $\supp(\mu) \subseteq L_+$. Then equality holds true in \eqref{diagonal-equ} for every point $x\in\supp (\mu)$ and thus by integrating, we obtain that equality holds in \eqref{diagonal-int} as well. 

Conversely, let us assume, that 
$$d_{W_1}( \delta_{(-1,-1)},\mu) + d_{W_1}(\mu, \delta_{(1,1)}) = 2$$ for some measure $\mu \in \woq$. We have to show that $\supp(\mu) \subseteq L_+$. We argue by contradiction: assume that the exists a point $x_0 \in\supp( \mu)$ that is not contained in $L_+$. Then there exists a small radius $r>0$ and a small $\varepsilon >0$ with the property that $\delta = \mu(B(x_0, r)>0 $ and 
$$ \dm((-1,-1), x) + \dm(x,(1,1)) >2 +\varepsilon , \ \textrm{for all} \ x \in B(x_0,r).$$
Using this relation we obtain 
\begin{align*}
 &d_{W_1}( \delta_{(-1,-1)},\mu) + d_{W_1}(\mu, \delta_{(1,1)}) =  \\&  = \int\limits_{\supp(\mu)}\dm((-1,-1), x )\ ~\mathrm{d}\mu(x) + \int\limits_{\supp(\mu)} \dm(x,(1,1)) \ ~\mathrm{d}\mu(x) = \\ & =  \int\limits_{\supp(\mu)}[\dm((-1,-1), x ) +  \dm(x, (1,1))] \  ~\mathrm{d}\mu(x) = \\ & = \int\limits_{\supp(\mu) \cap B(x_0,r)}[\dm((-1,-1), x ) +  \dm(x,(1,1)) ]\  ~\mathrm{d}\mu(x) + \\ & \hspace{1.5cm}+ \int\limits_{\supp(\mu) \setminus B(x_0,r)}[\dm((-1,-1), x ) +  \dm(x,(1,1) ) ] \  ~\mathrm{d}\mu(x)\\
 &> (2+\varepsilon) \delta + 2(1-\delta)= 2+\varepsilon\delta >2, 
\end{align*}
which is a contradiction. 
Let us consider an isometry $\Phi: \woq \to \woq$ that 
 fixes the elements of $V$, (i.e. the Dirac masses on the corners of $Q$). Then we have 
$$d_{W_1}(\delta_{(-1,-1)},\Phi(\mu))+ d_{W_1}(\Phi(\mu),\delta_{(1,1)}) = d_{W_1}(\delta_{(-1,-1)},\mu))+ d_{W_1}(\mu,\delta_{(1,1)}).$$
Assuming that $\supp(\mu) \subseteq L_+$ we obtain that by the above that 
$$ d_{W_1}(\delta_{(-1,-1)},\mu))+ d_{W_1}(\mu,\delta_{(1,1)})=2.$$ By the above equality we have, then 
 $$d_{W_1}(\delta_{(-1,-1)},\Phi(\mu))+ d_{W_1}(\Phi(\mu),\delta_{(1,1)})=2,$$
 which implies in turn that $\supp\big(\Phi(\mu)\big) \subseteq L_+$.

As mentioned at the beginning of the proof, the same argument shows that if we have $\supp(\mu) \subseteq L_-$, then it follows that 
$\supp\big(\Phi(\mu)\big) \subseteq L_-$ as well. Since $\mu_0=\delta_{(0,0)}$ has its support in $L_+ \cap L_-$, and it is the unique measure with this property,  we conclude, that $\Phi(\mu_0) = \mu_0$. 
Now let us consider the Wasserstein space $\mathcal{W}_1(L_+,\dm)$ of measures supported on $L_+$. By the above consideration, we have 
$$ \Phi:\mathcal{W}_1(L_+,\dm) \to \mathcal{W}_1(L_+,\dm), $$
moreover, we know that $\Phi(\mu_0) = \mu_0$ for the measure $\mu_0=\delta_{(0,0)}\in \mathcal{W}_1(L_+,\dm) $. We cannot apply the characterization of Wasserstein isometries on a line segment (see \cite[Theorem 2.5]{GTV1}) which implies that $\Phi(\mu) = \mu$ for all measures $\mu \in \mathcal{W}_1(L_+,\dm) $. In the same way we can also conclude, that $\Phi(\mu) = \mu$ for all measures $\mu \in \mathcal{W}_1(L_-,\dm) $.

\end{proof}

The combination of Proposition \ref{Diagonal-Rigidity} and Theorem \ref{T: main-3} implies that $\woq$ is isometrically rigid.

\subsection{Diagonal rigidity of $\wpq$ for $p>1$}
In this final subsection, we consider the case $X= Q= [-1,1]^2$ and  $p>1$. In this case, we also have the statement:

\begin{theorem}\label{T: main-4}
    The Wasserstein space $\wpq$ is diagonally rigid.

\end{theorem}
\begin{proof}
Let $\Phi$ be an isometry of the Wasserstein space $\wpq$. Since the Wasserstein space %$\mathcal{W}_p([0,2],d_{|\cdot|})$
$\mathcal{W}_p([-1,1],d_{|\cdot|})$
is isometrically rigid, Corollary \ref{corners} and the remark after that, implies that we can assume without loss of generality that $\Phi$ leaves every measure supported either on the right vertical side $[(1,-1),(1,1)]$ or on the top horizontal side $[(-1,1),(1,1)]$ fixed. We are going to show that $\Phi$ leaves every measure supported on the line diagonal segment $[(-1,-1),(1,1)] \subset Q$ fixed. The case of the other diagonal is analogous.

In the first step, we consider measures supported on the upper half of the diagonal, i.e. on the line segment $[(0,0),(1,1)]$. We prove first that these measures will be fixed. In the second step we proceed to consider general measures supported on the full segment $[(-1,-1),(1,1)]$. 

Let $\mu$ be a measure supported on the upper half of the diagonal of $Q$ (that is, on $[(0,0),(1,1)]$), and consider the projections $p_r: \, (t,t) \mapsto (1, 2t-1)$ and $p_u: (t,t) \mapsto (2t-1,1)$ that map $[(0,0),(1,1)]$ onto the right vertical side $[(1,-1),(1,1)]$ and the top horizontal side $[(-1,1),(1,1)],$ respectively. Let us introduce the notation 
$\mu_r:=(p_r)_{\#}(\mu)\quad\mbox{and}$ and $\mu_u:=(p_u)_{\#}(\mu)$.
\begin{figure}[H]
\centering
\includegraphics[width=0.8\textwidth]{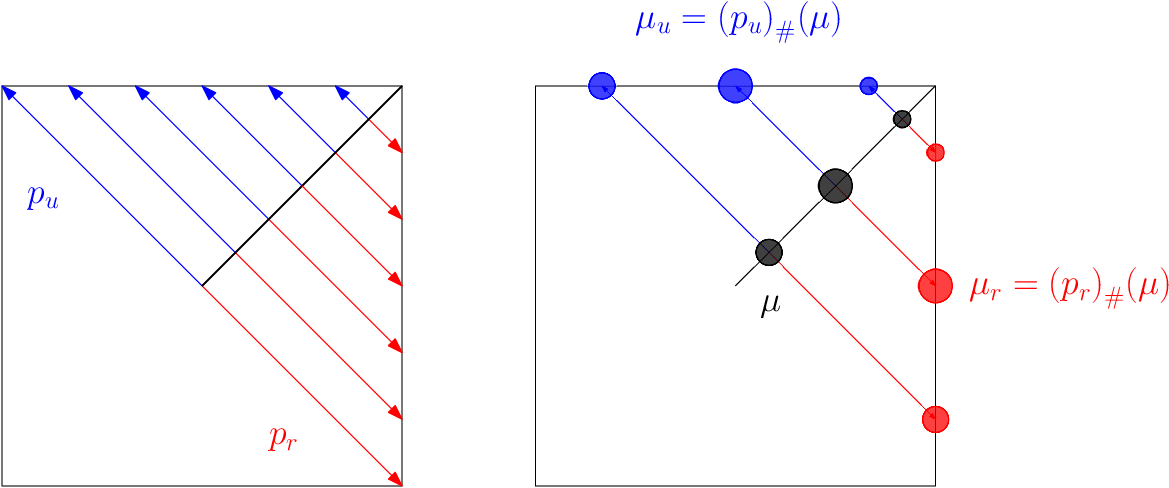}
\caption{Illustration for the definition of $p_u$, $p_r$, $\mu_u$ and $\mu_r$.}
\end{figure}

\par
Our goal is to show that $\mu$ is the unique minimizer of the functional
\begin{equation} \label{eq:p-power-error-functional}
\wpq \ni \nu \mapsto d_{W_p}^p(\mu_r, \nu)+d_{W_p}^p(\nu, \mu_u).    
\end{equation}
Once this is proved, we are done, as in this case the equation
\begin{equation}
\begin{split}
d_{W_p}^p(\mu_r, \Phi(\mu))+d_{W_p}^p(\Phi(\mu), \mu_u)
&=d_{W_p}^p(\Phi(\mu_r), \Phi(\mu))+d_{W_p}^p(\Phi(\mu), \Phi(\mu_u))\\
&=d_{W_p}^p(\mu_r,\mu)+d_{W_p}^p(\mu, \mu_u)\\
&=\min_{\nu \in \wpq}(d_{W_p}^p(\mu_r, \nu)+d_{W_p}^p(\nu, \mu_u))
\end{split}
\end{equation}
forces $\Phi(\mu)$ to be $\mu.$\\

Let $\nu \in \wpq$ be arbitrary, let $\pi_{(\mu_r,\nu)}^*$ be an optimal transport plan between $\mu_r$ and $\nu$ with respect to the cost $c(x,y)=\dm^p(x,y)$, let $\pi_{(\nu,\mu_u)}^*$ be an optimal transport plan between $\nu$ and $\mu_u$ with respect to the same cost, and let $\pi_{(\mu_r,\nu,\mu_u)}^* \in \mathrm{Prob}(Q^3)$ be the \emph{gluing} of them --- see the ``gluing lemma" \cite[Lemma 7.6]{V} for further details of this construction. Moreover, set $\pi_{(\mu_r,\mu_u)}^*:=\left( \pi_{(\mu_r,\nu,\mu_u)}^* \right)_{13}.$ Now
\begin{equation}\label{eq:1st-estimate}
\begin{split}
d_{W_p}^p(\mu_r, \nu)+d_{W_p}^p(\nu, \mu_u)
&=\iint_{Q^2}\dm^p(x,y) \dd \pi_{(\mu_r,\nu)}^*(x,y)
+\iint_{Q^2}\dm^p(y,z) \dd \pi_{(\nu,\mu_u)}^*(y,z)\\
&=\iiint_{Q^3} (\dm^p(x,y)+\dm^p(y,z)) \dd \pi_{(\mu_r,\nu,\mu_u)}^*(x,y,z)\\
&\geq \iiint_{Q^3} \min_{y \in Q}\{\dm^p(x,y)+\dm^p(y,z)\} \dd \pi_{(\mu_r,\nu,\mu_u)}^*(x,y,z)\\
&= \iint_{Q^2} \min_{y \in Q}\{\dm^p(x,y)+\dm^p(y,z)\} \dd \pi_{(\mu_r,\mu_u)}^*(x,z).
\end{split}
\end{equation}
Let us compute $\min_{y \in Q}\{\dm^p(x,y)+\dm^p(y,z)\}$ for any $(x,z) \in Q^2$ --- the case $x=z$ gives a trivial zero. By the reversed triangle inequality, we have
$$
\dm^p(x,y) \geq |\dm(x,z)-\dm(y,z)|^p=\dm^p(x,z)\left| 1 -\frac{\dm(y,z)}{\dm(x,z)}\right|^p.
$$
Consequently,
$$
\dm^p(x,y)+\dm^p(y,z) \geq \dm^p(x,z) \left( \left| 1 -\frac{\dm(y,z)}{\dm(x,z)}\right|^p + \left| \frac{\dm(y,z)}{\dm(x,z)}\right|^p\right).
$$
It is crucial that $p>1$ and hence the map $t \mapsto |t|^p$ is strictly convex on $\R.$ Therefore, the function $\R \ni \lambda \mapsto |1-\lambda|^p+|\lambda|^p$ has a unique minimizer which is $\lambda_0=\frac{1}{2},$ and the minimum is $2^{1-p}$ --- this can be justified by simple one-variable calculus. To sum up,
$$
\min_{y \in Q}\{\dm^p(x,y)+\dm^p(y,z)\}=2^{1-p} \dm^p(x,z),
$$
and this minimum is achieved if and only if $\dm(x,y)=\dm(y,z)=\frac{1}{2}\dm(x,z)$ --- note that this does not imply that $y=\frac{1}{2}(x+z).$ So we can continue \eqref{eq:1st-estimate} as follows:
\begin{equation}\label{eq:2nd-estimate}
\begin{split}
\iint_{Q^2} \min_{y \in Q}\{\dm^p(x,y)+\dm^p(y,z)\} \dd \pi_{(\mu_r,\mu_u)}^*(x,z)
&=2^{1-p} \iint_{Q^2} \dm^p(x,z) \dd \pi_{(\mu_r,\mu_u)}^*(x,z)\\ 
&\geq 2^{1-p} d_{W_p}^p(\mu_r,\mu_u).
\end{split}
\end{equation}
The inequality in \eqref{eq:2nd-estimate} is saturated if and only if $\pi_{(\mu_r,\mu_u)}^*$ is an \emph{optimal} transport plan between $\mu_r$ and $\mu_u$ with respect to the cost $c(x,z)=\dm^p(x,z).$ Observe, that $(p_r \times p_u)_{\#}(\mu)$ is the unique optimal transport plan between $\mu_r$ and $\mu_u$ for this cost. Indeed, by the definition of the max norm we get
\begin{equation} \label{eq:trivi-1}
(1-(2t-1))^p \leq \dm^p((1, 2t-1),(2s-1,1))
\end{equation}
and  
\begin{equation} \label{eq:trivi-2}
(1-(2s-1))^p \leq \dm^p((1, 2t-1),(2s-1,1))
\end{equation}
for all $(s,t) \in [0,1] \times [0,1].$
By the construction of $\mu_r$ and $\mu_u,$ the couplings of these measures are in $1-1$ correspondence with the couplings of $\mu$ with itself --- now we consider $\mu$ as a measure on $[0,1].$ This correspondence is described as follows: if $\pi \in C(\mu_r,\mu_u),$ then let us define $\tilde{\pi} \in C(\mu,\mu)$ by $\tilde{\pi}:=(p_r^{-1} \times p_u^{-1})_{\#} (\pi).$ Then
\begin{equation}\label{eq:push-1}
\begin{split}
\iint_{Q^2}\dm^p(x,z) \dd \pi(x,z)&=\iint_{[0,1]^2} \dm^p((1, 2t-1),(2s-1,1)) \dd \tilde{\pi}(t,s).\\
&\geq \iint_{[0,1]^2} (1-(2t-1))^p \dd \tilde{\pi}(t,s)\\
&=\int\limits_{[0,1]} (1-(2t-1))^p \dd \mu(t)
\end{split}
\end{equation}
and \eqref{eq:push-1} is saturated if and only if $s \geq t$ for $\tilde{\pi}-$a.e. $(s,t).$ Using \eqref{eq:trivi-2} we get
$$
\iint_{Q^2}\dm^p(x,z) \dd \pi(x,z) \geq \int\limits_{[0,1]} (1-(2s-1))^p \dd \mu(s)
$$
which is saturated if and only if $t \geq s$ for $\tilde{\pi}-$a.e. $(s,t).$ Therefore, if $\pi$ is an optimal coupling of $\mu_r$ and $\mu_u,$ then $\tilde{\pi}$ is supported on the diagonal of $[0,1]^2$ which implies that $\tilde{\pi}=(id \times id)_{\#}(\mu)$. This means that $\pi= (p_r \times p_u)_{\#}(\mu).$ So $(p_r \times p_u)_{\#}(\mu)$ is the only optimal transport plan.
\par
At this point we know by \eqref{eq:1st-estimate} and \eqref{eq:2nd-estimate} that 
\begin{equation} \label{eq:est-final}
d_{W_p}^p(\mu_r, \nu)+d_{W_p}^p(\nu, \mu_u) \geq 2^{1-p} d_{W_p}^p(\mu_r,\mu_u),
\end{equation}
and \eqref{eq:est-final} is saturated if and only if
$\pi_{(\mu_r,\mu_u)}^*=(p_r \times p_u)_{\#}(\mu)$ and $\dm(x,y)=\dm(y,z)=\frac{1}{2}\dm(x,z)$ for $\pi_{(\mu_r, \nu,\mu_u)}^*-$a.e. $(x,y,z) \in Q^3.$
Therefore, equality in \eqref{eq:est-final} implies that
$$
\mathrm{supp}(\pi_{(\mu_r,\nu,\mu_u)}^*)
\subset \left\{((1, 2t-1),y,(2t-1,1)) \, \middle| \, t \in [0,1], \, y \in Q\right\}.
$$
However, the unique metric midpoint of $(1, 2t-1)$ and $(2t-1,1)$ is $(t,t).$ Therefore, $y=\frac{1}{2}(x+z)$ must hold for $\pi_{(\mu_r,\nu,\mu_u)}^*$-a.e. $(x,y,z) \in Q^3,$ which forces $\nu$ to be 
$$
\left((x,z) \mapsto \frac{1}{2}(x+z) \right)_{\#} \pi_{(\mu_r,\mu_u)}^*
=\left((x,z) \mapsto \frac{1}{2}(x+z) \right)_{\#} (p_r \times p_u)_{\#}(\mu)= \mu.
$$

Now let us consider $\mu$, a probability measure supported on the main diagonal $[(-1,-1),(1,1)].$ The displacement interpolation given by
$$
\mu_s:=\left((x,x) \mapsto (1-s)(1,1)+s(x,x)\right)_{\#}(\mu) \qquad (s \in [0,1]) 
$$
is the unique geodesic line segment between $\mu_0=\delta_{(1,1)}$ and $\mu_1=\mu.$ Note that $\mu_{\frac{1}{2}}$ is supported on the ``upper half of the diagonal" $[(0,0),(1,1)]$ and hence preserved by $\Phi.$ Moreover, $\left(\mu_s\right)_{0\leq s\leq \frac{1}{2}}$ is the only geodesic line segment between $\delta_{(1,1)}$ and $\mu_{\frac{1}{2}},$ and the unique extension of this geodesic segment to the parameter domain $[0,1]$ is $\left(\mu_s\right)_{0\leq s\leq 1}.$ Therefore, the geodesic line segment $\left(\phi(\mu_s)\right)_{0\leq s\leq 1}$ containing $\Phi(\mu_0)=\Phi(\delta_{(1,1)})=\delta_{(1,1)}$ and $\Phi(\mu_{\frac{1}{2}})=\mu_{\frac{1}{2}}$ must coincide with $\left(\mu_s\right)_{0\leq s\leq 1},$ in particular, $\Phi(\mu)=\Phi(\mu_1)=\mu_1=\mu.$
\end{proof}

The combination of Proposition \ref{Diagonal-Rigidity} and the above theorem implies that the Wasserstein space $\wpq$ is isometrically rigid for $p>1$.

\section{Final remarks and open questions}

We think that the ideas developed in this paper can be used to prove isometric rigidity of Wasserstein spaces that are built over certain normed spaces. Recent studies of the structure of Wasserstein isometries feature interesting examples of both rigid and non-rigid Wasserstein spaces. Kloeckner showed in \cite{K} that the quadratic Wasserstein space over $\mathbb{R}^n$ admits non-trivial isometries. As a recent result of Che, Galaz-García, Kerin, and Santos-Rodríguez demonstrates \cite{S-R2}, non-trivial isometries show up even if the underlying normed space $X$ can be written as $H\times Y$, where $H$ is a Hilbert space and $Y$ is a finite-dimensional normed space. Furthermore, by an application of a recent result of Balogh, St\"oher, Titkos and Virosztek (see Theorem 1.1. in \cite{BSTV}) it follows that the Wasserstein space $\mathcal{W}_1(Q, d_1)$, (where $d_1$ is the $\ell_1$- metric on $Q$) is non-rigid as it contains mass-splitting isometries. This result is in sharp contrast to our main result, Theorem \ref{T: main}.  On the other hand, non-quadratic Wasserstein spaces over Hilbert spaces \cite{GTV1,GTV2}, and Wasserstein spaces over spheres, tori, and Heisenberg groups \cite{BTV,TnSn} turned out to be rigid in the last years. Based on these recent developments, it would be an intricate question to study the problem of rigidity of quadratic and non-quadratic Wasserstein spaces over general normed spaces.\\

The question of rigidity of quadratic Wasserstein spaces have been investigated for various underlying spaces including Hadamard spaces \cite{BK}, \cite{BK2} and more general metric spaces with negative curvature in the sense of Alexandrov. The case of positive sectional curvature has been considered by Santos-Rodriguez \cite{S-R}. It would be an interesting question to study the problem of rigidity in more general metric measure spaces satisfying a curvature-dimension condition, the so-called $CD(K,N)$ (see \cite{LV} and \cite{Villani}) for $K>0$  Note, that $CD(K,N)$ spaces have a generalized lower bound on the Ricci curvature and therefore are more general objects that the spaces considered in \cite{S-R}. 
\\
\paragraph*{{\bf Acknowledgments}} The authors wish to express their gratitude to Eric Str\"oher for pointing out a gap in one of our proofs and to the anonymous referee for numerous comments improving both content and presentation of the manuscript.
\\ 
\paragraph*{{\bf Declaration}} The authors declare that they have no competing interests.

\end{document}